\documentclass{article}

\usepackage[margin=1in,a4paper]{geometry}
\usepackage{authblk}

\usepackage[british]{babel}
\usepackage{microtype}
\usepackage[T1]{fontenc}
\usepackage{lmodern}

\usepackage[numbers,sort]{natbib}
\let\citet\Citet
\newcommand*{\doi}[1]{\href{http://dx.doi.org/#1}{\nolinkurl{doi:#1}}}

\usepackage{amsmath,amsthm,amsfonts,amssymb}
\usepackage{mathtools}

\mathtoolsset{showonlyrefs}

\allowdisplaybreaks
\newcommand{\trans}[1]{{#1}^\text{\normalfont T}} 
\newcommand{\dd}{\,\text{\normalfont d}} 
\newcommand{\vareps}{{\ensuremath\varepsilon}}
\newcommand{\varepsb}{{\ensuremath\boldsymbol\varepsilon}}
\DeclareMathOperator*{\argmin}{arg\,min}

\numberwithin{equation}{section}

\newtheorem{theorem}{Theorem}
\newtheorem{lemma}[theorem]{Lemma}
\newtheorem{example}{Example}

\usepackage{algorithm}
\usepackage{algorithmicx}
\usepackage{algpseudocode}

\usepackage{xcolor}

\usepackage{graphicx}
\graphicspath{{images/}}
\usepackage{epstopdf}
\usepackage[format=hang,labelfont={bf}]{caption}
\usepackage{subcaption}
\newcommand{\includepgf}[3]{\includegraphics[width=#1]{{figures/#3}.pdf}}
\let\oldsubcaption\subcaption
\renewcommand{\subcaption}[1]{\captionsetup{justification=centering}\oldsubcaption{#1}}

\usepackage{hyperref}

\title{A geometrical interpretation of the addition of nodes to an interpolatory quadrature rule while preserving positive weights}
\author[1,2]{L.\,M.\,M.~van~den~Bos\footnote{Corresponding author: \texttt{l.m.m.van.den.bos@cwi.nl}}}
\author[1]{B.~Sanderse}

\affil[1]{Centrum~Wiskunde~\&~Informatica, P.O.~Box 94079, 1090~GB, Amsterdam}
\affil[2]{Delft~University~of~Technology, P.O.~Box 5, 2600~AA, Delft}

\begin{document}
\maketitle

\begin{abstract}
\noindent A novel mathematical framework is derived for the addition of nodes to univariate and interpolatory quadrature rules. The framework is based on the geometrical interpretation of the Vandermonde matrix describing the relation between the nodes and the weights and can be used to determine all nodes that can be added to an interpolatory quadrature rule with positive weights such that the positive weights are preserved. In the case of addition of a single node, the derived inequalities that describe the regions where nodes can be added are explicit. Besides addition of nodes these inequalities also yield an algorithmic description of the replacement and removal of nodes. It is shown that it is not always possible to add a single node while preserving positive weights. On the other hand, addition of multiple nodes and preservation of positive weights is always possible, although the minimum number of nodes that need to be added can be as large as the number of nodes of the quadrature rule. In case of addition of multiple nodes the inequalities describing the regions where nodes can be added become implicit. It is shown that the well-known Patterson extension of quadrature rules is a special case that forms the boundary of these regions and various examples of the applicability of the framework are discussed. By exploiting the framework, two new sets of quadrature rules are proposed. Their performance is compared with the well-known Gaussian and Clenshaw--Curtis quadrature rules, demonstrating the advantages of our proposed nested quadrature rules with positive weights and fine granularity.
\end{abstract}
\begingroup
\endgroup

\section{Introduction}
This article is concerned with the addition of nodes to univariate and interpolatory quadrature rules with positive weights. If such a quadrature rule is given, the goal is to determine all sequences of nodes such that, upon adding all nodes from such a sequence to the rule, an interpolatory quadrature rule with positive weights is again obtained. The motivation of this problem is twofold. Firstly, approximations of integrals computed using interpolatory quadrature rules with positive weights converge for any absolute continuous function~\cite{Polya1933,Cools1997,Brass2011}. Secondly, nested quadrature rules allow for straightforward refinements of the quadrature rule approximation, which is especially relevant if the integrand is computationally expensive.

Possibly the best-known interpolatory quadrature rule is the Gaussian quadrature rule~\cite{Golub1969}, which exists for virtually any probability distribution with finite moments. It has positive weights and maximal polynomial degree. However, the nodes are not nested. The Gauss--Kronrod quadrature rule is an extension of a Gaussian quadrature rule, such that two nested rules with positive weights are obtained~\cite{Laurie1997,Vladislav2004}. The Gauss--Kronrod--Patterson quadrature rule~\cite{Monegato2001,Patterson1968} further extends this idea by repeatedly applying the same algorithm, such that a sequence of nested rules is obtained. However, it does not exist for any distribution~\cite{Kahaner1978,Kahaner1984}. Even though many other extensions have been proposed over the years~\cite{Laurie1996,Ma1996,Genz1996}, in general it is difficult to obtain a series of nested quadrature rules with positive weights. Moreover often the smallest possible granularity between two consecutive nested quadrature rules can only be found by exhaustive search~\cite{Bourquin2015}.

An other large group of well-known quadrature rules is formed by the Clenshaw--Curtis quadrature rules~\cite{Clenshaw1960}, or simply those quadrature rules that are based on Chebyshev approximations (the Clenshaw--Curtis rule is formed by the Chebyshev extrema). Besides having excellent interpolation properties~\cite{Ibrahimoglu2016}, it is well-known that these quadrature rules have positive weights if the distribution under consideration is uniform (explicit expressions are known~\cite{Trefethen2000a}). Moreover for non-uniform distributions, the condition number of the quadrature rule converges to unity~\cite{Brass2011}. However, the vanilla Clenshaw--Curtis nodes are only nested for exponentially growing numbers of nodes~\cite{Hasegawa1993}.

Both the Gaussian and Clenshaw--Curtis quadrature rules have explicitly predefined nodes based on the roots of orthogonal polynomials. This results into accurate quadrature rules, but the construction of an accurate nested quadrature rule with fine granularity based on these rules remains notoriously difficult.

In this article the goal is to propose a geometrical framework for the addition of nodes to an interpolatory quadrature rule with positive weights and use this framework to determine all interpolatory quadrature rules with positive weights that extend a rule based on predefined nodes. It will be demonstrated rigorously that the boundary of the set that contains all nodes that can be added is equivalent to the Patterson extension of quadrature rules, such that a special case of the framework is an extension of the aforementioned Gaussian quadrature rule families.

The approach taken is based on the geometrical interpretation of the linear system describing the nodes and the weights~\cite{Bos2016b,Peherstorfer1981,Dahlquist2008}, which yields a necessary and sufficient condition for a quadrature rule to have positive weights. The framework embeds previous results on the removal of nodes from quadrature rules~\cite{Bos2016b,Cools1989,Wilson1969} and describes, besides a geometrical description of all nodes that can be added to a quadrature rule, algorithms that can be used to construct and modify interpolatory quadrature rules with positive weights.

The addition and replacement of a single node can be determined analytically, whereas numerical methods are required to determine the bounds on the sets describing multiple nodes. The focus of this article is mainly on the geometrical and mathematical aspects and not on the numerical aspects of the proposed algorithms. However, to illustrate the potential of the framework, two straightforward examples of quadrature rules with positive weights that can be constructed by exploiting the proposed techniques are discussed.

In Section~\ref{sec:preliminaries} the nomenclature used in this article is discussed, including the motivation behind enforcing positive weights. In Section~\ref{sec:extensions1} the problem of adding a single node to a quadrature rule is considered, which can be solved analytically. It is not always possible to add a node to a quadrature rule such that the resulting rule has positive weights. Therefore the theory is extended to adding multiple nodes in Section~\ref{sec:extensionsn}, where the results developed for adding a single node will be used extensively. It is always possible to add multiple nodes to a quadrature rule, provided that any number of nodes may be added to the rule. To demonstrate the advantages of nested quadrature rules with positive weights, two quadrature rules that are derived in this work are compared with the well-known Gaussian and Clenshaw--Curtis quadrature rule. The details and results of this numerical experiment are discussed in Section~\ref{sec:numerics}. Conclusions and suggestions for future work are discussed in Section~\ref{sec:conclusion}.

\section{Preliminaries}
\label{sec:preliminaries}
The quadrature rule nomenclature relevant for this article is discussed in Section~\ref{subsec:nomenclature}. The relevance of positive weights and the relation between positive weights and accuracy of a quadrature rule is briefly reviewed in Section~\ref{subsec:accuracy}. The mathematical notion of adding nodes to a quadrature rule can be interpreted as a non-trivial extension of the removal of nodes~\cite{Bos2016b,Cools1989,Wilson1969}, which is briefly discussed in Section~\ref{subsec:removal}. Finally, the problem setting of this article and the main results obtained from this work are summarized mathematically in Section~\ref{subsec:probsetting}.

\subsection{Nomenclature}
\label{subsec:nomenclature}
A quadrature rule is a well-known approach to approximate a weighted integral in the interval $\Omega = [a, b] \subset \mathbb{R}$ with $-\infty \leq a < b \leq \infty$. The weighting function is a positive density function $\rho\colon \Omega \to [0, \infty)$. The main interest is to approximate the integral over a given continuous function $u\colon \Omega \to \mathbb{R}$, i.e.\ to approximate the following operator:
\begin{equation}
	\mathcal{I} u = \int_\Omega u(x) \, \rho(x) \dd x = \int_a^b u(x) \, \rho(x) \dd x.
\end{equation}
A quadrature rule approximates this integral by means of a weighted average, consisting of nodes and weights, which we denote by $X_N = \{x_0, \dots, x_N\} \subset \Omega$ and $W_N = \{w_0, \dots, w_N\} \subset \mathbb{R}$ respectively. The quadrature rule is the following operator $\mathcal{A}_N$:
\begin{equation}
	\mathcal{A}_N u \coloneqq \sum_{k=0}^N u(x_k) w_k \approx \mathcal{I} u.
\end{equation}

It is common to measure the consistency of this construction by means of polynomial degree. The polynomial degree of a quadrature rule is defined as the maximum polynomial degree the quadrature rule integrates exactly, or equivalently: a quadrature rule of degree $K$ has the property
\begin{equation}
	\label{eq:degree}
	\mathcal{A}_N \varphi = \mathcal{I} \varphi, \text{ for all $\varphi \in \mathbb{P}(K)$},
\end{equation}
where $\mathbb{P}(K)$ denotes the space of all univariate polynomials of degree $K$ or less. This definition is only meaningful if $\rho$ has finite moments, so that is assumed to be the case throughout this article.

A quadrature rule is called \emph{interpolatory} if the dimension of $\mathbb{P}(K)$ is larger than or equal to the number of nodes, or in other words, if $K \geq N$. Such quadrature rules can be formed by integrating the polynomial interpolant of $u$ using the nodes $X_N$. As the title of this article suggests, these quadrature rules are the main focus of this work: throughout this article the interest is mainly in rules with $K = N$ (though quadrature rules with $K > N$, such as the Gaussian quadrature rules, will also be considered).

The operators $\mathcal{A}_N$ and $\mathcal{I}$ and the space $\mathbb{P}(K)$ are linear, so if $K = N$, \eqref{eq:degree} defines a linear system that can be used to determine the weights, given the nodes and the moments of the distribution. Throughout this article a monomial basis of $\mathbb{P}(K)$ is considered. In this case, the matrix of the linear system is the well-known Vandermonde matrix, denoted as follows:
\begin{equation}
	\label{eq:momentmatch}
	\underbrace{\begin{pmatrix}
		x_0^0 & \cdots & x_N^0 \\
		\vdots & \ddots & \vdots \\
		x_0^N & \cdots & x_N^N
	\end{pmatrix}}_{V(X_N)}
	\begin{pmatrix}
		w_0 \\
		\vdots \\
		w_N
	\end{pmatrix}
	=
	\begin{pmatrix}
		\mu_0 \\
		\vdots \\
		\mu_N
	\end{pmatrix},
\end{equation}
with $\mu_k$ the raw moments of $\rho$:
\begin{equation}
	\mu_k = \int_\Omega x^k \, \rho(x) \dd x.
\end{equation}
Throughout this article it is assumed that $\mu_k$ is known exactly for all $k$. The notation $V(X_N)$ is used for the matrix of this linear system. It is well-known that
\begin{equation}
	\label{eq:determinant}
	\det V(X_N) = ~ \prod_{\mathclap{0 \leq i < j \leq N}} ~ (x_j - x_i),
\end{equation}
such that, given the nodes, \eqref{eq:momentmatch} defines a unique solution of the weights provided that all nodes are distinct.

\subsection{Accuracy of quadrature rules}
\label{subsec:accuracy}
In this article the focus is on constructing interpolatory quadrature rules with non-negative weights (which we will call with a little abuse of nomenclature a positive quadrature rule). An approximation of an integral by means of such a quadrature rule converges if the integrand is sufficiently smooth~\cite{Polya1933}, which can among others be demonstrated by applying the Lebesgue inequality~\cite{Brass2011}, provided that $\Omega$ is bounded. To this end, let $u$ be given and let $\varphi_N$ be the best approximation polynomial~\cite{Watson1980} of degree $N$ of $u$, i.e.\ $\varphi_N = \argmin_{\varphi \in \mathbb{P}(N)} \| \varphi - u \|_\infty$. Then
\begin{align}
	\label{eq:lebderiv}
	| \mathcal{A}_N u - \mathcal{I} u | &\leq (\| \mathcal{A}_N \|_\infty + \| \mathcal{I} \|_\infty) \| u - \varphi_N \|_\infty \\
	&= (\| \mathcal{A}_N \|_\infty + \mu_0) \| u - \varphi_N \|_\infty.
\end{align}
Here, it holds that
\begin{equation}
	\| \mathcal{A}_N \|_\infty = ~ \sup_{\mathclap{\| u \|_\infty = 1}} ~ | \mathcal{A}_N u | = \sum_{k=0}^N | w_k | = \sum_{k=0}^N w_k = \mu_0,
\end{equation}
where it is used that $|w_k| = w_k$. Hence the following inequality is obtained:
\begin{equation}
\label{eq:lebesgue}
	| \mathcal{A}_N u - \mathcal{I} u | \leq 2 \, \mu_0 ~ \inf_{\mathclap{\varphi \in \mathbb{P}(N)}} ~ \| u - \varphi \|_\infty.
\end{equation}
This shows many similarities with the classical Lebesgue inequality~\cite{Ibrahimoglu2016,Watson1980} and demonstrates that if $u$ can be approximated well using a polynomial, it can be integrated using a quadrature rule with positive weights. Similar results exist for unbounded $\Omega$~\cite{Brass2011,Zhou2003}.

Two well-known interpolatory quadrature rules with positive weights are the Clenshaw--Curtis and Gaussian quadrature rules.

The Clenshaw--Curtis quadrature rule~\cite{Clenshaw1960} has nodes $X_N$ that are defined as follows for $\Omega = [-1, 1]$:
\begin{equation}
	\label{eq:clenshaw}
	x_k = \cos \left( \frac{k}{N} \pi \right), \text{ for $k = 0, \dots, N$}.
\end{equation}
The Clenshaw--Curtis quadrature rule has positive weights if the uniform distribution is considered and for any other distribution with bounded support the sum of the absolute weights becomes arbitrary close to $\mu_0$ for large $N$~\cite{Brass2011}. The quadrature rule is nested for specific levels: it holds that $X_{N_L} \subset X_{N_{L+1}}$ with $N_L = 2^L$ (for $L = 1, 2, \dots$).

The nodes of the Gaussian quadrature rule~\cite{Golub1969} are defined as the roots of the orthogonal polynomials with respect to the distribution $\rho$ under consideration, e.g.\ Legendre polynomials for the uniform distribution, Jacobi polynomials for the Beta distribution, etc. The uniquely defined rules always have positive weights and with $N+1$ nodes the rule has degree $2N+1$, however the rules are not nested.

The Gauss--Kronrod and Gauss--Patterson quadrature rules are extensions of Gaussian quadrature rules such that upon adding $M$ nodes (with $M = N+2$ for the Gauss--Kronrod rule) to a rule of $N+1$ nodes, a (not necessarily positive) rule of degree $N+2M$ is obtained~\cite{Patterson1968,Laurie1997}. The Patterson extension is also applicable to non-Gaussian quadrature rules, though possibly complex-valued nodes can be obtained. The idea is to solve the following problem for $x_{N+1}, \dots, x_{N+M}$, given quadrature rule nodes $X_N$:
\begin{equation}
	\label{eq:patterson}
	\int_\Omega x^j \left[ \prod_{k=0}^{N+M} (x - x_k) \right] \rho(x) \dd x = 0, \text{ for $j = 0, \dots, M-1$}.
\end{equation}
Then the obtained rule has degree $N+2M$~\cite[Theorem~5.1.3]{Brass2011}, is defined uniquely, and possibly has complex-valued nodes. By construction, a Gaussian quadrature rule is obtained if $M=N+1$ (the weights of the nodes in $X_N$ become zero). These rules are reobtained as a special case in the framework discussed in this work.

\subsection{Removal of nodes}
\label{subsec:removal}
The primary focus of this article is on the \emph{addition} of nodes, but the obtained mathematical expressions can be interpreted as reversing the \emph{removal} of nodes from an existing quadrature rule. Using Carath\'eodory's theorem, it can be shown that for each positive interpolatory quadrature rule $X_N$, $W_N$ there exist two nodes $x_{k_0}$ and $x_{k_1}$ such that $X_N \setminus \{ x_{k_0} \}$ and $X_N \setminus \{ x_{k_1} \}$ both form the nodes of interpolatory quadrature rules with positive weights~\cite{Bos2016b,Cools1989,Wilson1969,Rabinowitz1986}. The details are discussed in the constructive proof of the following theorem.

\begin{theorem}[Carath\'eodory's theorem]
	\label{thm:caratheodory}
	Let $\mathbf{v}_0, \dots, \mathbf{v}_N$ be $N+1$ vectors spanning an $N$-dimensional space $V$. Let $\mathbf{v} \in V$ be such that $\mathbf{v} = \sum_{k=0}^N a_k \mathbf{v}_k$ with all $a_k \geq 0$. Then there exist non-negative $b_k$ and a $k_0 \in \{0, \dots, N\}$ such that
	\begin{equation}
		\mathbf{v} = \sum_{\substack{k = 0 \\ k \neq k_0}}^N b_k \mathbf{v}_k.
	\end{equation}
\end{theorem}
\begin{proof}
	The vectors $\mathbf{v}_0, \dots, \mathbf{v}_N$ are linearly dependent, since these are $N+1$ vectors spanning an $N$-dimensional space. Hence there exists a vector $\mathbf{c} = \trans{(c_0, \dots, c_N)} \neq \mathbf{0}$ such that
	\begin{equation}
		\sum_{k=0}^N c_k \mathbf{v}_k = \mathbf{0}.
	\end{equation}
	Hence for any $\alpha \in \mathbb{R}$, we have that
	\begin{equation}
		\mathbf{v} = \sum_{k=0}^N (a_k - \alpha c_k) \mathbf{v}_k.
	\end{equation}
	In particular, consider the following $\alpha$ and $k_0$:
	\begin{equation}
		\alpha = \min\left( \frac{a_k}{c_k} \mid c_k > 0 \right) \eqqcolon \frac{a_{k_0}}{c_{k_0}}.
	\end{equation}
	With these choices it holds that $a_k - \alpha c_k \geq 0$ for all $k$ and $a_{k_0} - \alpha c_{k_0} = 0$, concluding the proof as follows:
	\begin{align}
		\mathbf{v} &= \sum_{\substack{k = 0 \\ k \neq k_0}}^N (a_k - \alpha c_k) \mathbf{v}_k. \qedhere
	\end{align}
\end{proof}

The theorem can be used straightforwardly to remove nodes from a quadrature rule. To this end, let the positive interpolatory quadrature rule $X_N$ and $W_N$ be given. The goal is to construct an interpolatory quadrature rule using $N$ nodes from $X_N$ (which consists of $N+1$ nodes). Therefore, let $\mathbf{v}_0, \dots, \mathbf{v}_N$ be the columns of the Vandermonde matrix of degree $N-1$, i.e.\ $\mathbf{v}_k = \trans{(x_k^0, \dots, x_k^{N-1})}$. Then $\mathbf{v}_k$ are $N+1$ vectors spanning an $N$-dimensional space. The proof of Carath\'eodory's theorem yields that there exists a vector $\mathbf{c}$, scalar $\alpha$, and index $k_0$ such that
\begin{equation}
	\mu_j = \sum_{\substack{k = 0 \\ k \neq k_0}}^N x_k^j (w_k - \alpha c_k), \text{ for all $j = 0, \dots, N-1$}.
\end{equation}
Moreover, we have that $w_{k_0} - \alpha c_{k_0} = 0$, so by using $X_{N-1} = \{ x_k \in X_N \mid k \neq k_0 \}$ and $W_{N-1} = \{ w_k - \alpha c_k \mid k \neq k_0 \}$ a positive interpolatory quadrature rule is obtained. Notice that the vector $\mathbf{c}$ is computable, since it is a null vector of the Vandermonde matrix of degree $N-1$ (which is a $N \times (N+1)$-matrix).

This approach can be used to compute nested quadrature rules, but limits the accuracy of those quadrature rules to the initial rule of which nodes are removed. It is of less use if this rule is inadequately accurate or if no such rule is available. A possible approach to alleviate this is to use random samples as initial quadrature rule~\cite{Bos2018b}, though such samples do not accurately integrate higher order moments. The necessity of an existing quadrature rule is one of the main motivations to consider the \emph{addition} of nodes, since that does not require the computation of an initial quadrature rule of sufficient accuracy.

\subsection{Problem setting and main results}
\label{subsec:probsetting}
The problem studied in this article is how to add nodes to a positive interpolatory quadrature rule such that it remains positive and interpolatory. To formulate this mathematically, let a positive interpolatory quadrature rule $X_N$, $W_N$ be given. Then the goal is to determine, for given $M$, all nodes such that the set $X_{N+M}$ contains the nodes of a positive interpolatory quadrature rule and such that the rules are nested, i.e.\ $X_N \subset X_{N+M}$. To keep the nomenclature concise, we will refer to this problem as adding $M$ nodes to a positive interpolatory quadrature rule, where by ``adding'' we always mean addition such that the resulting quadrature rule has positive weights. Moreover the number of nodes added to a quadrature rule should be minimal, so we are also interested in the minimal value of $M$ (with $M > 0$) such that a positive interpolatory quadrature rule with nodal set $X_{N+M}$ exists.

If a positive quadrature rule is given that is not interpolatory, i.e.\ a quadrature rule such that $\mathcal{A}_N \varphi = \mathcal{I} \varphi$ for all $\varphi \in \mathbb{P}(K)$ with $K < N$, a positive interpolatory quadrature rule can be deduced from this rule by repeatedly applying Theorem~\ref{thm:caratheodory}. Therefore we assume in this article without loss of generality that all quadrature rules are interpolatory.

The approach is to formulate, for given $M$, a necessary and sufficient condition for all $M$ nodes that can be added. This condition can be used firstly to determine whether such nodes exist for a specific $M$ and secondly to determine the nodes themselves. Moreover the derived theory allows for specific adjustments of quadrature rules. These adjustments consist of replacing and removing nodes from the quadrature rule, in such a way that the degree of the rule is not affected.

The analysis is split into two sections. The addition of a single node ($M = 1$) can be solved analytically and is discussed in Section~\ref{sec:extensions1}. The addition of multiple nodes ($M > 1$) can only be done analytically for special cases. Based on the theory for $M = 1$, this problem is analyzed in Section~\ref{sec:extensionsn}.

\section{Addition of one node}
\label{sec:extensions1}
Let $X_N$, $W_N$ be a positive interpolatory quadrature rule. The goal is to determine all $x_{N+1}$ such that $X_{N+1} = X_N \cup \{ x_{N+1} \}$ forms the nodes of a positive interpolatory quadrature rule, i.e.\ there exists a set of non-negative weights $W_{N+1}$ such that
\begin{equation}
	\sum_{k=0}^{N+1} x_k^j w^{(N+1)}_k = \mu_j, \text{ for $j = 0, \dots, N+1$}.
\end{equation}
Here, $w^{(N+1)}_k$ are the weights in the set $W_{N+1}$ and $\mu_j$ is assumed to be known. Notice that in general $W_N$ and $W_{N+1}$ will completely differ, so we use the following notation for any $N$:
\begin{equation}
	W_N = \{w^{(N)}_0, \dots, w^{(N)}_N\}.
\end{equation}
Moreover, with a little abuse of notation we will use $w^{(N)}_k = 0$ for all $k > N$.

In Section~\ref{subsec:positivity1} we derive a necessary and sufficient condition for such an $x_{N+1}$ to exist, which depends on the current nodes, weights, and moment $\mu_{N+1}$. As such, the developed theory provides practical adjustments of a quadrature rule. These constitute addition and replacement of a node, without reducing the degree of the interpolatory quadrature rule. The details are discussed in Section~\ref{subsec:adjustments1} and will be very useful in the remainder of this article. In Section~\ref{subsec:construction1} the Patterson extension is discussed in light of the derived adjustments and some basic applications of the derived procedures are discussed, including the construction of a (partially) nested quadrature rule with positive weights.

\subsection{Positive weight criterion}
\label{subsec:positivity1}
The key notion is that if the node $x_{N+1}$ is given, a vector $\mathbf{c} = \trans{(c_0, \dots, c_{N+1})}$ can be constructed such that $w^{(N+1)}_k = w^{(N)}_k + c_k$ (for $k = 0, \dots, N+1$). This is the vector used in Section~\ref{subsec:removal} to remove nodes from a rule. If this vector is such that $c_k \geq -w_k^{(N)}$, then $w_k^{(N+1)} \geq 0$, which is the primary goal. In this section, these properties are translated to conditions on $x_{N+1}$ that describe in which cases a node can be added to a quadrature rule.

The interpolatory quadrature rule $X_N$, $W_N$ has degree $N$, so after adding $x_{N+1}$ the following should hold to ensure that the new rule is interpolatory:
\begin{equation}
	\mu_j = \sum_{k=0}^N x_k^j w^{(N)}_k = \sum_{k=0}^{N+1} x_k^j w^{(N+1)}_k, \text{ for $j = 0, \dots, N$}.
\end{equation}
From this, it follows for $j = 0, \dots, N$ that (using $w_{N+1}^{(N)} = 0$):
\begin{equation}
	\label{eq:c1}
	0 = \sum_{k=0}^{N+1} x_k^j w^{(N+1)}_k - \sum_{k=0}^{N+1} x_k^j w^{(N)}_k = \left(\sum_{k=0}^{N+1} x_k^j w^{(N)}_k + \sum_{k=0}^{N+1} x_k^j c_k\right) - \sum_{k=0}^{N+1} x_k^j w^{(N)}_k = \sum_{k=0}^{N+1} x_k^j c_k.
\end{equation}
The goal is to construct $X_{N+1}$ and $W_{N+1}$ such that they form a quadrature rule of degree $N+1$. Hence with $\mu_{N+1} = \int_\Omega x^{N+1} \, \rho(x) \dd x$ given, it should hold that
\begin{equation}
	\sum_{k=0}^{N+1} x_k^{N+1} w^{(N+1)}_k = \mu_{N+1},
\end{equation}
which can be expressed in terms of the vector $\mathbf{c}$ as
\begin{equation}
	\label{eq:defeps}
	\varepsilon_{N+1} \coloneqq \mu_{N+1} - \sum_{k=0}^N x_k^{N+1} w^{(N)}_k = \sum_{k=0}^{N+1} x_k^{N+1} c_k.
\end{equation}
The value of $\varepsilon_{N+1}$ can be interpreted as the approximation error of the quadrature rule with nodes $X_N$ and weights $W_N$ with respect to $\mu_{N+1}$. Combining \eqref{eq:c1} and \eqref{eq:defeps} yields the following system of linear equations for the vector $\mathbf{c}$:
\begin{equation}
	\begin{pmatrix}
		x_0^0 & \cdots & x_N^0 & x_{N+1}^0 \\
		\vdots & \ddots & \vdots & \vdots \\
		x_0^N & \cdots & x_N^N & x_{N+1}^N \\
		x_0^{N+1} & \cdots & x_N^{N+1} & x_{N+1}^{N+1}
	\end{pmatrix}
	\begin{pmatrix}
		c_0 \\
		\vdots \\
		c_N \\
		c_{N+1}
	\end{pmatrix}
	=
	\begin{pmatrix}
		0 \\
		\vdots \\
		0 \\
		\varepsilon_{N+1}
	\end{pmatrix},
\end{equation}
or more compactly:
\begin{equation}
	V(X_{N+1}) \, \mathbf{c} = \varepsb,
\end{equation}
with $\varepsb = \trans{(0, \dots, 0, \varepsilon_{N+1})}$. The vector $\varepsb$ has a large number of zeros so it is convenient to apply Cramer's rule to this linear system, which yields
\begin{equation}
	\label{eq:c_k}
	c_k = \frac{\det V_k(X_{N+1})}{\det V(X_{N+1})},
\end{equation}
where $V_k(X_{N+1})$ is equal to $V(X_{N+1})$ with the $k$-th column replaced by $\varepsb$, where the indexing of columns is started with 0. This expression can be simplified by noticing that
\begin{equation}
	\det V_k(X_{N+1}) = (-1)^{(N+2) + (k+1)} \vareps_{N+1} \det V(X_{N+1} \setminus \{ x_k \}) = (-1)^{N+k+1} \varepsilon_{N+1} \det V(X_{N+1} \setminus \{ x_k \}),
\end{equation}
with $V(X_{N+1} \setminus \{ x_k \})$ the $(N+1) \times (N+1)$ Vandermonde matrix constructed with the nodal set $X_{N+1} \setminus \{ x_k \}$. By using~\eqref{eq:determinant}, the following is obtained for $k = 0, \dots, N+1$:
\begin{equation}
\label{eq:uni1nodederiv}
\begin{aligned}
	c_k = \frac{\det V_k(X_{N+1})}{\det V(X_{N+1})} &= (-1)^{N+k+1} \varepsilon_{N+1} \frac{\det V(X_{N+1} \setminus \{ x_k \})}{\det V(X_{N+1})} \\
	&= (-1)^{N+k+1} \varepsilon_{N+1} \left.\middle( \prod_{\substack{0 \leq i < j \leq N+1 \\ i, j \neq k}} (x_j - x_i) \middle) \middle/ \middle( \prod_{0 \leq i < j \leq N+1} (x_j - x_i) \middle)\right. \\
	&= \left.\varepsilon_{N+1} \middle/ \middle( \prod_{\substack{j=0 \\ j \neq k}}^{N+1} (x_k - x_j) \middle)\right..
\end{aligned}
\end{equation}
The denominator of this expression can be written as $\ell'_N(x_k)$, where $\ell_N(x) = \prod_{j=0}^N (x - x_j)$ is the nodal polynomial. To keep the dependence on $x_{N+1}$ clear, this notation is used sparingly in this article.

The goal is to have positive weights, i.e.\ $w^{(N+1)}_k = w^{(N)}_k + c_k \geq 0$, which can be used to prove the following theorem.

\begin{theorem}
	\label{thm:uni1node}
	Let $X_N$, $W_N$ form an interpolatory quadrature rule. Then $X_{N+1} = X_N \cup \{ x_{N+1} \}$ forms the nodal set of a positive interpolatory quadrature rule if and only if
	\begin{equation}
		\label{eq:uni1node}
		\left. -\vareps_{N+1} \middle/ \middle( \prod_{\substack{j = 0 \\ j \neq k}}^{N+1} (x_k - x_j) \middle)\right. \leq w^{(N)}_k, \text{ for $k = 0, \dots, N+1$}.
	\end{equation}
\end{theorem}
\begin{proof}
If $X_{N+1}$ forms the nodal set of a positive interpolatory quadrature rule, then
\begin{equation}
	0 \leq w_k^{(N+1)} = w_k^{(N)} + c_k = w_k^{(N)} + \left.\varepsilon_{N+1} \middle/ \middle( \prod_{\substack{j=0 \\ j \neq k}}^{N+1} (x_k - x_j) \middle)\right..
\end{equation}
Subtracting $w_k^{(N)}$ from both sides of the inequality yields \eqref{eq:uni1node}. Vice versa, if \eqref{eq:uni1node} holds, it follows that
\begin{align}
	&w_k^{(N+1)} = w_k^{(N)} + c_k = w_k^{(N)} + \underbrace{\left.\varepsilon_{N+1} \middle/ \middle( \prod_{\substack{j=0 \\ j \neq k}}^{N+1} (x_k - x_j) \middle)\right.}_{= -w_k^{(N)}} = 0. \qedhere
\end{align}
\end{proof}

If $\vareps_{N+1} = 0$, i.e.\ $\mathcal{A}_N x^{N+1} = \mu_{N+1}$, then the theorem yields that the new rule has positive weights if and only if the current rule has positive weights. This is not surprising: any node $x_{N+1}$ can be added to such a rule with $w^{(N+1)}_{N+1} = 0$ (and with $w^{(N+1)}_k = w^{(N)}_k$ for $k = 0, \dots, N$).

From a computational point of view \eqref{eq:uni1node} might not be a numerically stable way of computing the bounds that describe all nodes that can be added. In the context of quadrature rules, numerical instabilities are usually alleviated by changing the basis of the Vandermonde matrix, but this is not applicable in this case since the determinant is up to a scaling factor independent from the basis used to construct the Vandermonde matrix (and this factor cancels out in \eqref{eq:c_k}). Nonetheless, \eqref{eq:uni1node} can be evaluated in a numerical stable way using the well-known barycentric formulation of the interpolating polynomial. The interested reader is referred to~\cite{Berrut2004}.

\subsection{Quadrature rule adjustments}
\label{subsec:adjustments1}
Theorem~\ref{thm:uni1node} describes a necessary and sufficient condition for a quadrature rule to have positive weights if both $x_{N+1}$ and $\vareps_{N+1}$ are known. A main novelty of this work is to employ a geometrical interpretation of \eqref{eq:uni1node}, from which several possible adjustments of quadrature rules can be derived. The most straightforward one is that all nodes $x_{N+1}$ can be determined that yield a positive interpolatory quadrature rule upon adding one of them to an existing quadrature rule. Moreover the formula also yields procedures to replace nodes in a quadrature rule, keeping the weights positive. The latter adjustment will be useful in Section~\ref{sec:extensionsn}, where it can be used to determine all possible $M$ nodes that can be added to a rule.

In Section~\ref{subsubsec:geometry1} we further consider \eqref{eq:uni1node} and discuss the geometrical relation between the new node $x_{N+1}$ and the quadrature error $\vareps_{N+1}$. In Section~\ref{subsubsec:addition1} and~\ref{subsubsec:replacement1} we discuss the addition and replacement of nodes in a positive interpolatory quadrature rule such that positivity of the weights is preserved. These operations follow directly from the geometrical interpretation of Theorem~\ref{thm:uni1node} derived in Section~\ref{subsubsec:geometry1}. The removal of a node, as outlined in Section~\ref{subsec:removal}, can also be formulated as a consequence of Theorem~\ref{thm:uni1node}, which is not done here since the removal of nodes has been considered extensively in previous work~\cite{Bos2016b,Cools1989,Wilson1969,Rabinowitz1986}.

\subsubsection{Geometry of nodal addition}
\label{subsubsec:geometry1}
The inequalities from \eqref{eq:uni1node} are $N+2$ linear inequalities in $x_{N+1}$ and $\vareps_{N+1}$. This can be seen by rewriting \eqref{eq:uni1nodederiv} as follows:
\begin{equation}
	\label{eq:uni1nodederiv2}
	c_k \prod_{\substack{j = 0 \\ j \neq k}}^{N+1} (x_k - x_j) = \varepsilon_{N+1}, \text{ for $k = 0, \dots, N+1$}.
\end{equation}
If two values of $x_{N+1}$, $c_k$ (for $k = 0, \dots, N+1$), or $\vareps_{N+1}$ are known, all other values can be determined from these expressions, which enforces that the obtained quadrature rule is again interpolatory. To incorporate positive weights, we use that for $k = 0, \dots, N$ it holds that
\begin{equation}
	\vareps_{N+1} = c_k \prod_{\substack{j = 0 \\ j \neq k}}^{N+1} (x_k - x_j) = (x_k - x_{N+1}) \, c_k \underbrace{\prod_{\substack{j = 0 \\ j \neq k}}^N (x_k - x_j)}_{\mathclap{\text{Independent from $x_{N+1}$}}}.
\end{equation}
By combining this with \eqref{eq:uni1node} and requiring $w^{(N)}_k + c_k \geq 0$ inequalities of the following form are obtained:
\begin{equation}
\label{eq:posineq1}
\begin{aligned}
	\vareps_{N+1} &\leq -w^{(N)}_k (x_k - x_{N+1}) \prod_{\substack{j = 0 \\ j \neq k}}^N (x_k - x_j) &\text{if } \prod_{\substack{j = 0 \\ j \neq k}}^{N+1} (x_k - x_j) \leq 0, \\
	\vareps_{N+1} &\geq -w^{(N)}_k (x_k - x_{N+1}) \prod_{\substack{j = 0 \\ j \neq k}}^N (x_k - x_j) &\text{if } \prod_{\substack{j = 0 \\ j \neq k}}^{N+1} (x_k - x_j) \geq 0.
\end{aligned}
\end{equation}
These are linear inequalities describing the relation between $x_{N+1}$ and $\vareps_{N+1}$ such that $w^{(N+1)}_k \geq 0$ for $k = 0, \dots, N$. For $k = N+1$ it holds that $w_k^{(N)} = 0$, so by using that $c_k = w_k^{(N+1)}$, \eqref{eq:uni1nodederiv2} translates to:
\begin{equation}
\label{eq:posineq2}
\begin{aligned}
	\vareps_{N+1} \leq 0 & &\text{if } \prod_{j=0}^N (x_{N+1} - x_j) \leq 0, \\
	\vareps_{N+1} \geq 0 & &\text{if } \prod_{j=0}^N (x_{N+1} - x_j) \geq 0.
\end{aligned}
\end{equation}
Even though the rightmost inequalities are non-linear, their sign solely depends on the location of $x_{N+1}$ with respect to the other nodes. Hence the exact value of the product is not of importance.

\begin{figure}[t]
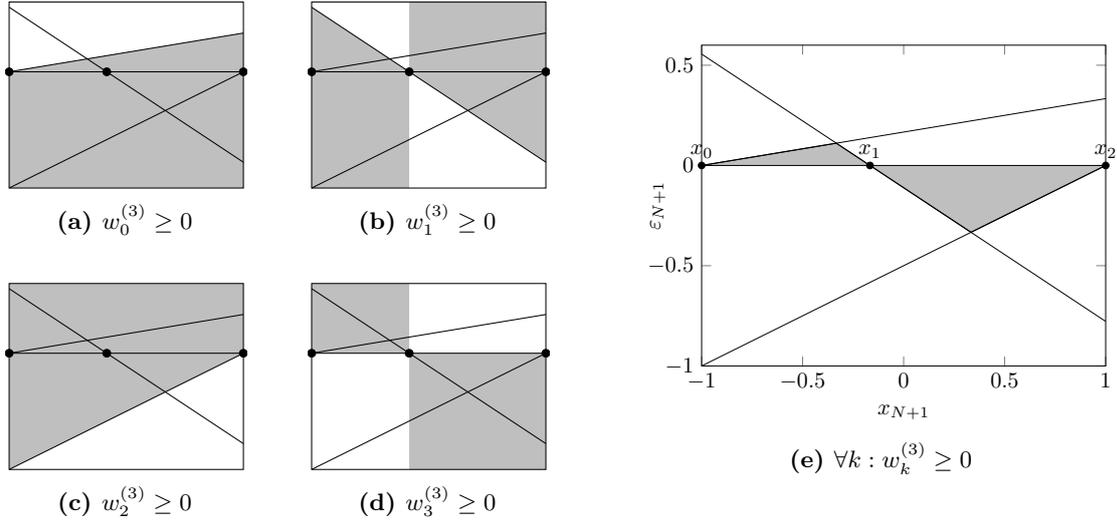

	\centering

	\begin{minipage}{.5\textwidth}
		\begin{minipage}{.5\textwidth}
			\centering
			\includepgf{.8\textwidth}{.64\textwidth}{ineq1.tikz}
			\subcaption{$w^{(3)}_0 \geq 0$}\label{fig:ineq1}
		\end{minipage}%
		\begin{minipage}{.5\textwidth}
			\centering
			\includepgf{.8\textwidth}{.64\textwidth}{ineq2.tikz}
			\subcaption{$w^{(3)}_1 \geq 0$}\label{fig:ineq2}
		\end{minipage}

		~\\[0.75ex]

		\begin{minipage}{.5\textwidth}
			\centering
			\includepgf{.8\textwidth}{.64\textwidth}{ineq3.tikz}
			\subcaption{$w^{(3)}_2 \geq 0$}\label{fig:ineq3}
		\end{minipage}%
		\begin{minipage}{.5\textwidth}
			\centering
			\includepgf{.8\textwidth}{.64\textwidth}{ineq4.tikz}
			\subcaption{$w^{(3)}_3 \geq 0$}\label{fig:ineq4}
		\end{minipage}
	\end{minipage}%
	\begin{minipage}{.5\textwidth}
		\centering
		\includepgf{.8\textwidth}{.64\textwidth}{ineq.tikz}
		\subcaption{$\forall k : w^{(3)}_k \geq 0$}\label{fig:ineqall}
	\end{minipage}%
	
	\caption{The quadrature rule error $\vareps_{N+1}$ versus the new node $x_{N+1}$ using the quadrature rule $X_N = \{ -1, -1/6, 1 \}$ and $\rho \equiv 1/2$. The solid lines depict pairs $(x_{N+1}, \vareps_{N+1})$ such that one weight becomes zero, after addition of $x_{N+1}$ to the quadrature rule using $\vareps_{N+1}$ as quadrature error. \emph{Left:} regions where individual weights are positive; the axes are labeled similar as the large rightmost figure. For example, if $(x_{N+1}, \vareps_{N+1})$ is picked in the gray region of subfigure~\protect\ref{fig:ineq2}, adding $x_{N+1}$ to the quadrature rule yields a rule with $w^{(3)} \geq 0$ (assuming $\vareps_{N+1}$ defines the raw moment correctly). \emph{Right:} region where all weights are positive, which is the intersection of the left figures. Hence if $(x_{N+1}, \vareps_{N+1})$ is picked in the gray region of subfigure~\protect\ref{fig:ineqall}, adding $x_{N+1}$ to the quadrature rule yields a rule with positive weights.}
	\label{fig:ineq}
\end{figure}

\begin{example}
	\label{ex:addition}
	The inequalities from~\eqref{eq:posineq2} are visualized as functions from $x_{N+1}$ to $\vareps_{N+1}$ in Figure~\ref{fig:ineq} for the quadrature rule with $X_N$ and $W_N$ as follows:
	\begin{equation}
		X_N = \left\{ -1, -\frac{1}{6}, 1 \right\}, W_N = \left\{ \frac{1}{10}, \frac{24}{35}, \frac{3}{14} \right\}.
	\end{equation}
	This is an (obviously positive) interpolatory quadrature rule with $\Omega = [-1, 1]$ and $\rho \equiv 1/2$. The solid lines in the figures depict all $(x_{N+1}, \vareps_{N+1})$ pairs such that one weight becomes equal to zero (i.e.\ where equality is attained in inequality \eqref{eq:posineq1} or \eqref{eq:posineq2}). The region where individual weights are positive are shaded in subfigures~\ref{fig:ineq1}, \ref{fig:ineq2}, \ref{fig:ineq3}, and \ref{fig:ineq4}. Subfigure~\ref{fig:ineqall} is the intersection of these figures and therefore depicts regions where all weights are positive. Any $(x_{N+1}, \vareps_{N+1})$ pair in the shaded region describes a positive interpolatory quadrature rule that contains the original three nodes.

	The left subfigures demonstrate some key properties of the derived inequalities. The inequalities are linear and switch sign at the node, which is the rightmost condition of \eqref{eq:posineq1}. The characteristics of the last inequality (subfigure~\ref{fig:ineq4}) solely depend on the location of $x_{N+1}$ with respect to the other nodes. A combination of all inequalities (subfigure~\ref{fig:ineqall}) has varying characteristics between different nodes, but it is always a system of linear inequalities. The line $\vareps_{N+1} = 0$ is contained in all shaded regions, because any node with weight equal to zero can be added to the rule if the next moment $\mu_{N+1}$ is already correctly integrated by the quadrature rule.
\end{example}

The relation between $x_{N+1}$ and $\vareps_{N+1}$ from \eqref{eq:posineq2} can be interpreted in two ways. Firstly, if a new node $x_{N+1}$ is given, an upper bound and lower bound on $\vareps_{N+1}$ can be determined such that upon adding $x_{N+1}$ to the quadrature rule, a positive interpolatory quadrature rule is obtained. Geometrically these are the bounds of the shaded area with the $x = x_{N+1}$ line. This interval is never empty (as $\vareps_{N+1} = 0$ is always in the shaded region). Secondly, if $\vareps_{N+1}$ is given, a (possibly empty) set can be determined such that a positive interpolatory quadrature rule is obtained upon adding a node from such a set. Geometrically this is equivalent to determining the bounds of the shaded area with the $y = \vareps_{N+1}$ line.

The second interpretation can be used to add nodes to a quadrature rule, i.e.\ $\vareps_{N+1}$ is known and the goal is to determine $x_{N+1}$ (this is discussed in Section~\ref{subsubsec:addition1}). The first interpretation can be used to replace nodes within a quadrature rule: $x_{N+1}$ is added to the node and an existing node can be removed by setting its weight to zero (this is discussed in Section~\ref{subsubsec:replacement1}).

\subsubsection{Addition of a node}
\label{subsubsec:addition1}
A direct consequence of \eqref{eq:posineq1} is that all nodes that can be added to a quadrature rule can be defined by means of intervals, obtained via a linear inequality. The results are discussed in the following lemmas. The first focuses on keeping the \emph{existing} weights of the quadrature rule positive, the second focuses on ensuring that the \emph{additional} weight (i.e.\ of the added node) is positive.

\begin{lemma}
	\label{lmm:posineq1}
	Let $X_N$, $W_N$ form the nodes and the weights of a positive interpolatory quadrature rule, let $\vareps_{N+1}$ from \eqref{eq:defeps} be given, and let index $k$ of node $x_k$ be given. Let $x_{N+1}^{[k]}$ be as follows:
	\begin{equation}
		x_{N+1}^{[k]} = \left(\vareps_{N+1} + w_k^{(N)} x_k \ell_N'(x_k)\middle) \middle/ \middle(w_k^{(N)} \ell_N'(x_k)\right).
	\end{equation}
	Then $w_k^{(N+1)} \geq 0$ if and only if $x_{N+1} \in I_k$ with
	\begin{equation}
		I_k = \mathbb{R} \setminus [x_k, x_{N+1}^{[k]}) \text{ if $x_k < x_{N+1}^{[k]}$, and } I_k = \mathbb{R} \setminus (x_{N+1}^{[k]}, x_k] \text{ otherwise}.
	\end{equation}
	Or in other words, if and only if $x_{N+1}$ is not between $x_k$ and $x_{N+1}^{[k]}$.
\end{lemma}
\begin{proof}
	Adding a node is determining an $x_{N+1}$ that solves \eqref{eq:posineq1} if $\vareps_{N+1}$ is known. Hence, to keep the $k$-th weight positive, this is equivalent to computing the solution $x^{[k]}_{N+1}$ of the following problem:
	\begin{equation}
		\label{eq:intervalbdd1}
		\vareps_{N+1} = -w_k^{(N)} (x_k - x^{[k]}_{N+1}) \underbrace{\prod_{\substack{j = 0 \\ j \neq k}}^N (x_k - x_j)}_{\ell_N'(x_k)},
	\end{equation}
	Here we used $\ell_N'$ to make the notation more compact. Hence if $w_k^{(N)} \neq 0$:
	\begin{equation}
		x^{[k]}_{N+1} = \frac{\vareps_{N+1} + w_k^{(N)} x_k \ell_N'(x_k)}{w_k^{(N)} \ell_N'(x_k)}.
	\end{equation}
	The node $x^{[k]}_{N+1}$ is such that, if added to the quadrature rule, an interpolatory quadrature rule is obtained with $w^{(N+1)}_k = 0$ (the other weights may be negative). Assume $x_k < x_{N+1}^{[k]}$, without loss of generality. Then any node $x_{N+1}$ with $x_{N+1} \geq x_{N+1}^{[k]}$ or $x_{N+1} < x_k$ solves \eqref{eq:posineq1} for a single $k$. This is equivalent to stating that $x_{N+1} \in I_k = \mathbb{R} \setminus [x_k, x_{N+1}^{[k]})$.
\end{proof}

The proof of this lemma can also be stated geometrically, using one of the Figures~\ref{fig:ineq1}, \ref{fig:ineq2}, or \ref{fig:ineq3}. If $\vareps_{N+1}$ is known, those $x_{N+1}$ that are such that $(x_{N+1}, \vareps_{N+1})$ is not part of a gray region form the interval as stated in the theorem. Here, $x_{N+1}^{[k]}$ is the intersection of the line passing through $x_k$ and the constant line $\vareps_{N+1}$. All intervals $I_k$ are bounded, so there always exists a node $x_{N+1} \in ( I_0 \cap \cdots \cap I_N )$, or in other words, there always exists a node that keeps the existing $N+1$ weights of a quadrature rule positive upon addition.

Obviously, the goal is also to ensure that the weight of the added node is positive, which can be described by means of a series of intervals. The details of this are discussed in the following lemma.

\begin{lemma}
	\label{lmm:posineq2}
	Let $X_N$, $W_N$ form the nodes and the weights of a positive interpolatory quadrature rule, let $\vareps_{N+1}$ from \eqref{eq:defeps} be given. Without loss of generality, assume that $x_0 < x_1 < \cdots < x_N$. Then $w_{N+1}^{(N+1)} \geq 0$ upon addition of $x_{N+1}$ to the quadrature rule if and only if one of the following holds for all $k = 0, \dots, N$:
	\begin{itemize}
		\item $x_{N+1} \in [x_{k-1}, x_k]$ if the signs of $\ell_N'(x_k)$ and $\vareps_{N+1}$ are equal (e.g.\ both are negative);
		\item $x_{N+1} \in [x_k, x_{k+1}]$ if the signs of $\ell_N'(x_k)$ and $\vareps_{N+1}$ differ.
	\end{itemize}
	For $k = N$, use $x_{k+1} = \infty$ and for $k = 0$, use $x_{k-1} = -\infty$ (with a little abuse of notation).
\end{lemma}
\begin{proof}
	Recall the derivation of \eqref{eq:posineq2}, i.e.\ the relation between $x_{N+1}$, $w_{N+1}^{(N+1)}$, and $\vareps_{N+1}$:
	\begin{equation}
		-w^{(N+1)}_{N+1} \prod_{j = 0}^N (x_{N+1} - x_j) = \vareps_{N+1}.
	\end{equation}
	It holds that $w^{(N+1)}_{N+1} > 0$ if $\prod_{j = 0}^N (x_{N+1} - x_j)$ and $\vareps_{N+1}$ have different sign. The first term flips sign only at $x_{N+1} = x_k$ (for any $k = 0, \dots, N$), hence if, for given $k$,
	\begin{equation}
		\prod_{\substack{j = 0 \\ j \neq k}}^N (x_k - x_j) > 0,
	\end{equation}
	it is necessary that $x_{k-1} < x_{N+1} < x_k$ to ensure that $\prod_{j=0}^N (x_{N+1} - x_j)$ is negative and $x_k < x_{N+1} < x_{k+1}$ to ensure that $\prod_{j=0}^N (x_{N+1} - x_j)$ is positive. A similar result holds if
	\begin{equation}
		\prod_{\substack{j = 0 \\ j \neq k}}^N (x_k - x_j) < 0.
	\end{equation}
	Combining this with the sign of $\vareps_{N+1}$ results in the statement of the lemma.
\end{proof}

Geometrically, Lemma~\ref{lmm:posineq2} describes the intervals of Figure~\ref{fig:ineq4}. Notice that Lemma~\ref{lmm:posineq2} can also straightforwardly be applied to cases where $w_k^{(N)} = 0$ (for any $k = 0, \dots, N$), i.e.\ if the quadrature rule has weights equal to zero.

Using Lemma~\ref{lmm:posineq1} and Lemma~\ref{lmm:posineq2} the set $I$ can be computed such that any $x_{N+1} \in I$ can be added to a quadrature rule $X_N$ and $W_N$ such that positive weights are obtained (and adding any $x_{N+1} \notin I$ yields a rule with at least one negative weight). The procedure is to firstly compute all intervals $I_0, \dots, I_N$ from Lemma~\ref{lmm:posineq1} and construct $I = I_0 \cup \cdots \cup I_N$. Secondly, Lemma~\ref{lmm:posineq2} is used to remove intervals of the form $[x_{k-1}, x_k]$ from $I$.

\begin{algorithm}[t]
\caption{Addition of a node}
\label{alg:addition}
\begin{algorithmic}[1]
\Require Positive, interpolatory quadrature rule $X_N, W_N$, raw moment $\mu_{N+1}$ (or, equivalently, $\vareps_{N+1}$)
\Ensure Set $I \subset \mathbb{R}$ such that $X_N \cup \{ x \}$ forms the nodes of a positive, interpolatory quadrature rule if and only if $x \in I$

~

\State $I \gets \mathbb{R}$
\State $\vareps_{N+1} \gets \mu_{N+1} - \sum_{k=0}^N x_k^{N+1} w_k^{(N)}$
\State Sort $X_N, W_N$ such that $x_0 < x_1 < \cdots < x_N$
\For{$k = 0, \dots, N+1$}
	\State $\ell_N'(x_k) \gets \prod_{j \neq k}^N (x_k - x_j)$

	~

	\If{$w_k^{(N)} > 0$}
		\State $x_{N+1}^{[k]} \gets \left(\vareps_{N+1} + w_k^{(N)} x_k \ell_N'(x_k)\middle) \middle/ \middle(w_k^{(N)} \ell_N'(x_k)\right)$
		\If{$x_{N+1}^{[k]} > x_k$}
			\State $I \gets I \setminus [ x_k, x_{N+1}^{[k]} )$
		\Else
			\State $I \gets I \setminus ( x_{N+1}^{[k]}, x_k ]$
		\EndIf
	\EndIf

	~

	\If{$(\ell_N'(x_k) < 0 \text{ and } \vareps_{N+1} < 0) \text{ or } (\ell_N'(x_k) > 0 \text{ and } \vareps_{N+1} > 0)$} \label{alg:ifstart}
		\If{$k > 0 \text{ and } w_k > 0$}
			\State $I \gets I \setminus [ x_{k-1}, x_k ]$
		\Else
			\State $I \gets I \setminus ( -\infty, x_k ]$
		\EndIf
	\Else
		\If{$k < N \text{ and } w_k > 0$}
			\State $I \gets I \setminus [ x_k, x_{k+1} ]$
		\Else
			\State $I \gets I \setminus [ x_k, \infty )$
		\EndIf
	\EndIf \label{alg:ifend}
\EndFor
\State \textbf{Return} $I$
\end{algorithmic}
\end{algorithm}

The exact details of this procedure are outlined in Algorithm~\ref{alg:addition}. No advanced interval arithmetic is necessary to implement this algorithm, only a procedure that implements the removal of an interval from a series of intervals is needed.

\begin{example}
	\label{ex:addition2}
	Reconsider the quadrature rule from Example~\ref{ex:addition}. Then the bounds of the intervals containing nodes that can be added, i.e.\ the solutions of \eqref{eq:intervalbdd1}, are depicted in Figure~\ref{subfig:addition} as open circles. Here, $\mu_{N+1} = 0$, so from a straightforward computation it follows that $\vareps_{N+1} = -1/9$. A constant $\rho$ is considered here. In this case, the values of $x^{[k]}_{N+1}$ are (from left to right) $-5/3$, $0$, and $7/9$, of which the first is not visible in the figure. Adding any of these nodes yields a quadrature rule with positive weights, but we emphasize that this is generally not the case for other quadrature rules. Hence adding any node from the set $I = (-\infty, -5/3] \cup [0, 7/9]$ yields a positive interpolatory quadrature rule. Restricting $x_{N+1}$ to the set $\Omega$ further reduces the number of possible intervals.
\end{example}

Notice that $I = \emptyset$ if $\vareps_{N+1} \neq 0$ and $w_k^{(N)} = 0$. This can be derived mathematically, but it also follows from the mere fact that all weights change (see \eqref{eq:intervalbdd1}) upon addition of a node to a quadrature rule, so $w^{(N+1)}_k = w^{(N)}_k = 0$ is not possible. If $\vareps_{N+1} = 0$, no node can be added to enforce that $w_k^{(N)} = 0$. However, any node with weight equal to zero can be added, hence the formula yields $x^{[k]}_{N+1} = x^{}_k$ with $w^{(N+1)}_{N+1} = 0$. Technically, the quadrature rule now has a node equal to $x_k$ with weight equal to zero. Nonetheless, this results into a singular Vandermonde matrix (which contradicts the theory developed so far), so we do not further study this specific case.

If $\Omega = \mathbb{R}$ and the number of nodes is odd, it is always possible to add a single node to a quadrature rule: in this case the result from Lemma~\ref{lmm:posineq2} either states that $x_{N+1} \in (-\infty, x_0]$ or $x_{N+1} \in [x_N, \infty)$, but never both. Geometrically this means that the leftmost and rightmost shaded region grow to infinity and minus infinity respectively (or vice versa). Similarly, if $\Omega = \mathbb{R}$ and the number of nodes is even, it is always possible to add a single node if $\vareps_{N+1} \geq 0$.

However, in any other case (i.e.\ that of a bounded $\Omega$ or even number of nodes with $\vareps_{N+1} < 0$) adding a single node to a quadrature rule is not always possible, as shown in the following example.
\begin{example}
	Adding a single node to the following interpolatory quadrature rule is not possible when requiring positive weights:
	\begin{equation}
		X_N = \left\{ -1, -\frac{1}{6}, \frac{1}{11}, 1 \right\}, W_N = \left\{ \frac{29}{180}, \frac{144}{595}, \frac{1331}{3060}, \frac{17}{105} \right\}.
	\end{equation}
	Note that this example can be obtained straightforwardly by adding the node $1/11$ to the quadrature rule from Example~\ref{ex:addition} and redetermining the weights likewise.
\end{example}

\begin{figure}
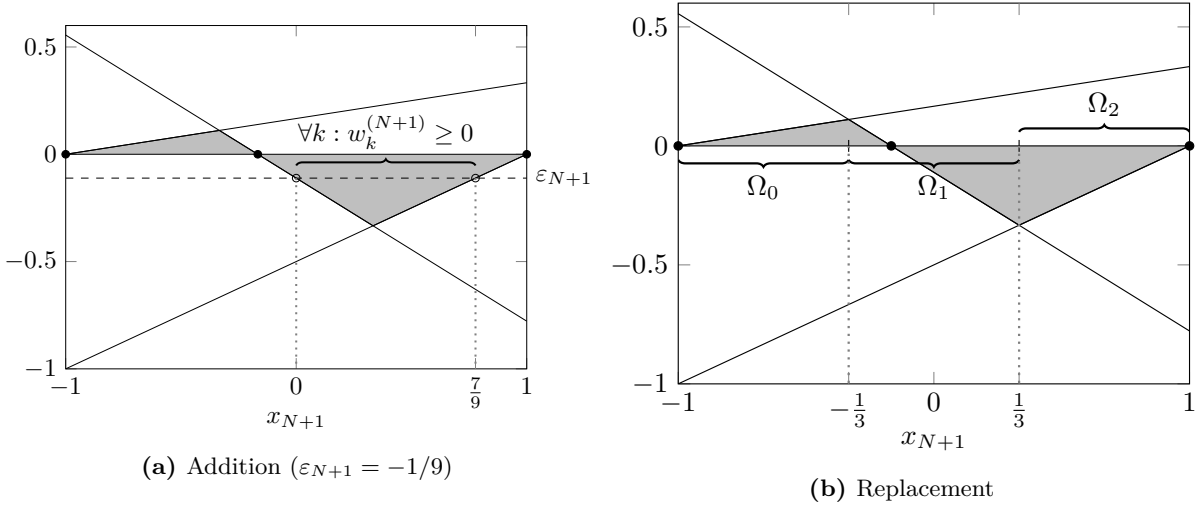

	\begin{minipage}{.5\textwidth}
		\centering
		\includepgf{\textwidth}{.8\textwidth}{addition.tikz}
		\subcaption{Addition ($\vareps_{N+1} = -1/9$)}
		\label{subfig:addition}
	\end{minipage}%
	\begin{minipage}{.5\textwidth}
		\centering
		\includepgf{\textwidth}{.8\textwidth}{replacement.tikz}
		\subcaption{Replacement}
		\label{subfig:replacement}
	\end{minipage}
	\caption{Addition of a new node to and replacement of an existing node within the quadrature rule $X_N = \{-1, -1/6, 1\}$ and $\rho \equiv 1/2$. \emph{Left:} all nodes that can be added to a quadrature rule form intervals, in this case the interval $[0, 7/9]$ and the interval $(-\infty, -5/3]$ (of which the latter is not depicted). \emph{Right:} the closed sets $\Omega_k$ depict all possible replacements within a quadrature rule. If the goal is to construct a positive interpolatory quadrature rule, the node $x_k$ can only be replaced by nodes from the set $\Omega_k$.}
	\label{fig:additionreplacement}
\end{figure}

\subsubsection{Replacement of a node}
\label{subsubsec:replacement1}
Replacing a node is equivalent to adding a node, with the difference that the goal is to determine this node such that the weight of an existing node in the obtained quadrature rule becomes zero, i.e.\ $w_k^{(N+1)} = 0$ for a $k \leq N$. This is equivalent to determining a specific $(x_{N+1}, \vareps_{N+1})$ pair that yields $w_k^{(N+1)} = 0$, which was used to determine all possible additions in Section~\ref{subsubsec:addition1}. The main difference with addition is that the next moment $\mu_{N+1}$ is not used, as the number of nodes and the degree of the rule do not change. This makes $\vareps_{N+1}$ a free variable.

The relation between $\vareps_{N+1}$ and $x_{N+1}$ is already derived, so by reconsidering \eqref{eq:intervalbdd1} with the goal to determine both $x_{N+1}$ and all $\vareps^{[k]}_{N+1}$ (indexed by $[k]$ with $k = 0, \dots, N$) that make $w_k^{(N+1)} = 0$ the following expressions are obtained:
\begin{equation}
	\label{eq:epsk}
	\vareps^{[k]}_{N+1} = -w_k^{(N)} (x_k - x_{N+1}) \prod_{\substack{j = 0 \\ j \neq k}}^N (x_k - x_j), \text{ for $k = 0, \dots, N$}.
\end{equation}
We will interpret this expression as a function of $x_{N+1}$, denoted by $\vareps^{[k]}_{N+1}\colon \Omega \to \mathbb{R}$. By using $\vareps_{N+1} = \vareps^{[k]}_{N+1}(x_{N+1})$, a positive interpolatory quadrature rule with $w^{(N+1)}_k = 0$ is obtained upon adding $x_{N+1}$ to the rule. 

It follows that for every $x_{N+1} \in \Omega$ there is an $x_k \in X_N$ such that the quadrature rule with nodes $(X_N \cup \{ x_{N+1} \}) \setminus \{ x_k \}$ is positive and interpolatory. The details are discussed in the following lemma.

\begin{lemma}
	\label{lmm:replacement1}
	Let $X_N$, $W_N$ form the nodes and the weights of a positive interpolatory quadrature rule and let $x_{N+1} \in \Omega$ be given. Then there exists an $x_k$ such that $(X_N \cup \{ x_{N+1} \}) \setminus \{ x_k \}$ forms the nodal set of a positive and interpolatory quadrature rule.
\end{lemma}
\begin{proof}
	Let $\vareps^{[k]}_{N+1}$ be defined by \eqref{eq:epsk}. Consider $\vareps_-$ and $\vareps_+$, defined as follows:
	\begin{align}
		\vareps_- &= \max_k \left( \vareps^{[k]}_{N+1} \mid \vareps^{[k]}_{N+1} < 0 \right), \\
		\vareps_+ &= \min_k \left( \vareps^{[k]}_{N+1} \mid \vareps^{[k]}_{N+1} > 0 \right).
	\end{align}
	Hence $\vareps_- < 0 < \vareps_+$. Using Lemma~\ref{lmm:posineq1}, it follows that using either $\vareps_-$ or $\vareps_+$ to add $x_{N+1}$ results in a quadrature rule with $w_k^{(N+1)} \geq 0$ for $k = 0, \dots, N$. Moreover, by definition of $\vareps_-$ and $\vareps_+$ these rules have one (or more) weights equal to 0. From Lemma~\ref{lmm:posineq2} it follows that either the rule constructed using $\vareps_-$ or $\vareps_+$ has $w_{N+1}^{(N+1)} \geq 0$ (and the other has $w_{N+1}^{(N+1)} \leq 0$).

	Concluding, either $\vareps_-$ or $\vareps_+$ can be used to construct a positive interpolatory quadrature rule with at least one weight equal to zero. Nodes with weights equal to zero can be removed without affecting the quadrature rules. This is equivalent to having added a node $x_{N+1}$ and removed one, say $x_k$, which is the statement of the theorem.
\end{proof}

\begin{algorithm}[t]
\caption{Replacement of a new node}
\label{alg:replacement1}
\begin{algorithmic}[1]
\Require Positive, interpolatory quadrature rule $X_N, W_N$, new node $x \not\in X_N$
\Ensure Positive, interpolatory quadrature rule $\hat{X}_N, \hat{W}_N$, with $x \in \hat{X}_N$ and $\#(\hat{X}_N \cap X_N) = N$

~

\State $\vareps_+ \gets \infty$
\State $\vareps_- \gets -\infty$
\State $k_+, k_- \gets -1$
\For{$k = 0, \dots, N$}
	\State $\vareps_{N+1}^{[k]} \gets -w_k^{(N)} (x_k - x_{N+1}) \prod_{j \neq k}^N (x_k - x_j)$
	\State $\ell'_{N+1} = (x - x_j) \prod_{j \neq k}^N (x_k - x_j)$

	~

	\If{$\ell'_{N+1} \leq 0 \text{ and } \vareps_+ > \vareps_{N+1}^{[k]} > 0$}
		\State $k_+ \gets k$
		\State $\vareps_+ \gets \vareps_{N+1}^{[k]}$
	\EndIf
	\If{$\ell'_{N+1} \geq 0 \text{ and } \vareps_- < \vareps_{N+1}^{[k]} < 0$}
		\State $k_- \gets k$
		\State $\vareps_- \gets \vareps_{N+1}^{[k]}$
	\EndIf
\EndFor

~

\If{$\prod_{j=0}^N (x - x_j) > 0$}
	\State $c_k \gets \left. \vareps_+ \middle/ \middle( (x - x_j) \prod_{j \neq k}^N (x_k - x_j) \right)$ (for $k = 0, \dots, N$)
	\State $c \gets \left. \vareps_+ \middle/ \prod_{j = 0}^N (x - x_j) \right.$
	\State $k_0 \gets k_+$
\Else
	\State $c_k \gets \left. \vareps_- \middle/ \middle( (x - x_j) \prod_{j \neq k}^N (x_k - x_j) \right)$ (for $k = 0, \dots, N$)
	\State $c \gets \left. \vareps_- \middle/ \prod_{j = 0}^N (x - x_j) \right.$
	\State $k_0 \gets k_-$
\EndIf

~

\State $\hat{X}_N \gets \{ x_0, \dots, x_{k_0-1}, x, x_{k_0+1}, \dots, x_N \}$
\State $\hat{W}_N \gets \{ w^{(N)}_0 + c_0, \dots, w^{(N)}_{k_0-1} + c_{k_0-1}, c, w^{(N)}_{k_0+1} + c_{k_0+1}, \dots, w^{(N)}_N + c_N \}$
\State \textbf{Return} $\hat{X}_N, \hat{W}_N$

\end{algorithmic}
\end{algorithm}

The proof of the lemma is constructive, and therefore describes a straightforward method to replace nodes in a quadrature rule. Given $x_{N+1} \in \Omega$, the procedure is to compute $\vareps_+$ and $\vareps_-$, figure out whether $\vareps_{N+1} = \vareps_+$ or $\vareps_{N+1} = \vareps_-$ yields $w^{(N+1)}_{N+1} \geq 0$ by using using Lemma~\ref{lmm:posineq2}, and finally compute the quadrature rule after replacement. These steps are outlined in detail in Algorithm~\ref{alg:replacement1}. Geometrically, the approach computes the two lines closest to the $\vareps_{N+1} = 0$ line, i.e.\ the boundary of the gray region, and determines which of these lines corresponds to obtaining a quadrature rule with only positive weights (see Figure~\ref{subfig:replacement}).

Consequently, the domain of a quadrature rule, depicted by $\Omega \subset \mathbb{R}$, can be decomposed in subsets $\Omega_0, \dots, \Omega_N$ that indicate which node can be replaced. If $x_{N+1} \in \Omega_k$, $(X_N \cup \{ x_{N+1} \}) \setminus \{ x_k \}$ forms the nodes of a positive interpolatory quadrature rule. Combining the results of Lemma~\ref{lmm:posineq2} and Lemma~\ref{lmm:replacement1}, these sets can be denoted in the following way:
\begin{equation}
	x_{N+1} \in \Omega_k \iff
	\left\{
	\begin{aligned}
		\vareps^{[k]}_{N+1} &= \min_j\left(\vareps^{[j]}_{N+1} \mid \vareps^{[j]}_{N+1}(x_{N+1}) \geq 0\right) &\text{if } \prod_{j \neq k}^{N+1} (x_k - x_j) \leq 0, \\
		\vareps^{[k]}_{N+1} &= \max_j\left(\vareps^{[j]}_{N+1} \mid \vareps^{[j]}_{N+1}(x_{N+1}) \leq 0\right) &\text{if } \prod_{j \neq k}^{N+1} (x_k - x_j) \geq 0.
	\end{aligned}
	\right. 
\end{equation}

The sets $\Omega_k$ have been depicted in Figure~\ref{subfig:replacement}. Notice that the boundaries of these sets correspond to positions where two lines intersect, or in other words, those $x_{N+1} \in \Omega_k$ that result into two weights equal to zero, if used for replacement. One of these weights is, by construction, $w_k^{(N+1)}$. If the other weight is $w_l^{(N+1)}$, we also have $x_{N+1} \in \Omega_l$. This geometrical observation can be made explicit, which can be used to actually compute $\Omega_k$: these $x_{N+1}$ have $\vareps_{N+1}^{[k]}(x_{N+1}) = \vareps_{N+1}^{[l]}(x_{N+1})$, or equivalently:
\begin{equation}
	-w_k^{(N)} (x_k - x_{N+1}) \prod_{\substack{j = 0 \\ j \neq k}}^N (x_k - x_j) = -w_l^{(N)} (x_l - x_{N+1}) \prod_{\substack{j = 0 \\ j \neq l}}^N (x_l - x_j).
\end{equation}
Hence we have proved the following lemma.

\begin{lemma}
	\label{lmm:replacement3}
	Let $k$ be given and let $\partial \Omega_k$ denote the boundary of $\Omega_k$. Then, for any $x_{N+1} \in \partial \Omega_k$, we have that
	\begin{equation}
		-w_k^{(N)} (x_k - x_{N+1}) \prod_{\substack{j = 0 \\ j \neq k}}^N (x_k - x_j) = -w_l^{(N)} (x_l - x_{N+1}) \prod_{\substack{j = 0 \\ j \neq l}}^N (x_l - x_j),
	\end{equation}
	for an $l \in 0, \dots, N$.
\end{lemma}

The result is a procedure to compute the boundaries of a specific $\Omega_k$. Firstly, for $l = 0, \dots, N$, compute $x_{(k,l)}$ such that
\begin{equation}
	\label{eq:xkl}
	-w_k^{(N)} (x_k - x_{(k,l)}) \prod_{\substack{j = 0 \\ j \neq k}}^N (x_k - x_j) = -w_l^{(N)} (x_l - x_{(k,l)}) \prod_{\substack{j = 0 \\ j \neq l}}^N (x_l - x_j).
\end{equation}
Those $x_{(k,l)}$ that yield a positive interpolatory quadrature rule upon replacement (e.g.\ computed using Algorithm~\ref{alg:replacement1}), form the boundary of the interval $\Omega_k$. If $x_l < x_{(k,l)}$, it follows that $[x_l, x_{(k,l)}) \notin \Omega_k$, since a replacement with $x_{N+1} \in [x_l, x_{(k,l)}]$ results in a negative $w^{(N+1)}_l$ (similarly for $x_l > x_{(k,l)}$). The procedure to determine $\Omega_k$ explicitly is outlined in Algorithm~\ref{alg:replacement2}. Here, the indexing is slightly changed to be able to reuse parts of Algorithm~\ref{alg:addition}, since we still need to ensure that the weight of the added node (which replaces $x_k$) is positive.

\begin{algorithm}[t]
\caption{Replacement of a given node}
\label{alg:replacement2}
\begin{algorithmic}[1]
\Require Positive, interpolatory quadrature rule $X_N, W_N$, node $x_l \in X_N$
\Ensure Space $\Omega_l$, such that $(X_N \cup \{ x \}) \setminus \{ x_l \}$ forms the nodes of a positive, interpolatory quadrature rule if and only if $x \in \Omega_l$ 

~

\State $\Omega_l \gets \mathbb{R}$
\State $\ell_N'(x_l) \gets \prod_{j \neq l}^N (x_l - x_j)$
\For{$k = 0, \dots, l-1, l+1, \dots, N$}
	\State $\ell_N'(x_k) \gets \prod_{j \neq k}^N (x_k - x_j)$
	\State $x_{(k,l)} \gets \left( w^{(N)}_k x_k \ell_N'(x_k) - w^{(N)}_l x_l \ell_N'(x_l) \middle) \middle/ \middle( w_k \ell_N'(x_k) - w_l \ell_N'(x_l) \right)$

	~

	\If{$x_k < x_{(k,l)}$}
		\State $\Omega_l \gets \Omega_l \setminus [ x_k, x_{(k,l)} )$
	\Else
		\State $\Omega_l \gets \Omega_l \setminus ( x_{(k,l)}, x_k ]$
	\EndIf

	~

	\State $\vareps_{N+1} \gets -w_l (x_l - x_{(k,l)}) \ell_N'(x_l)$
	\State Follow steps \ref{alg:ifstart}--\ref{alg:ifend} of Algorithm~\ref{alg:addition}
\EndFor

\end{algorithmic}
\end{algorithm}

Equation \eqref{eq:xkl} does not necessarily have a solution for any $l$. Geometrically this is the case if the lines through $x_k$ and $x_l$ are parallel. In such a case, one should use $x_{(k,l)} \gets \infty$ or $x_{(k,l)} \gets -\infty$ in Algorithm~\ref{alg:replacement2}, depending on the sign of the nominator when computing $x_{(k,l)}$ (usually, this happens automatically when using floating point arithmetic).

The values of $x_{N+1}$ that solve \eqref{eq:xkl} form a special case. Since $x_{N+1} \in \Omega_k \cap \Omega_l$, the quadrature rule $(X_N \cup \{ x_{N+1} \}) \setminus \{ x_k, x_l \}$ is positive, interpolatory, and has degree $N$, even though it consists only of $N$ nodes. The latter result is remarkable: two nodes are removed and one is added, but the degree of the quadrature rule is not affected. Such rules have a non-trivial high degree and are therefore more accurate than interpolatory quadrature rules without this property.

\begin{example}
	An example of an interpolatory quadrature rule with non-trivial high degree is $X_N = \{ -1, 1/3 \}$, obtained by adding $1/3$ to the quadrature rule of Example~\ref{ex:addition} (and removing all nodes with zero weight). All nodes that can be added to obtain such a rule are the intersection of two lines in Figure~\ref{subfig:replacement}.

	More generally, all nodes $x_k$ that can be added to a rule can be found by determining the bounds on the shaded region and observing which node belongs to the obtained bound. Consequently, the fact that $\Omega = \bigcup_{k=0}^N \Omega_k$ follows visually from Figure~\ref{subfig:replacement}. Hence the relation between $\vareps^{[k]}_{N+1}$ and $x_{N+1}$, as described by \eqref{eq:epsk}, are the solid lines in Figure~\ref{subfig:replacement}. 
\end{example}

The node $x_{(k,l)}$ only depends on the nodes $x_j$ with $j \neq k$ and $j \neq l$, i.e.\ its value is independent from $x_k$ and $x_l$. This is not evident, as \eqref{eq:xkl} depends on these nodes. However, it can be demonstrated by using that the rule interpolatory, which yields:
\begin{equation}
	w^{(N)}_k = \int_\Omega L_k(x) \, \rho(x) \dd x = \frac{1}{\ell'_N(x_k)} \int_\Omega \frac{\ell_N(x)}{x - x_k} \rho(x) \dd x, \text{ with } L_k(x) = \prod_{\substack{j = 0 \\ j \neq k}}^N \frac{x - x_j}{x_k - x_j}.
\end{equation}
Here, $L_k(x)$ is the $k$-th Lagrange basis polynomial. Replacing this expression in \eqref{eq:xkl} and using that $\ell'_N(x_k) = \prod_{j \neq k} (x_k - x_j)$ yields an equality that can be simplified to the following:
\begin{equation}
	x_{(k,l)} = \left( \int_\Omega x \, \ell_{(k,l)}(x) \, \rho(x) \dd x \middle) \middle/ \middle( \int_\Omega \ell_{(k,l)}(x) \, \rho(x) \dd x \right), \text{ with } \ell_{(k,l)}(x) = \prod_{\substack{j = 0 \\ \mathclap{j \neq k,l}}}^N (x - x_j).
\end{equation}
This expression is in fact a Patterson extension (consider \eqref{eq:patterson} with $j=0$). The tight relation between the Patterson extension and the framework discussed in this article is further discussed in Section~\ref{subsubsec:patterson1}.

\subsection{Constructing quadrature rules}
\label{subsec:construction1}
In the previous section the theoretical foundation for extending a positive interpolatory quadrature rule with a single node is derived. In this section, firstly it is discussed how addition relates naturally to the Patterson extension~\cite{Patterson1968,Patterson1989} of (non-Gaussian) quadrature rules. Secondly, due to the simplicity of addition and replacement of a node, quadrature rules based on these procedures can be derived numerically fast and accurately, and an example is discussed.

As discussed previously, there does not always exist a single node that can be added such that positive weights are obtained, so it is non-trivial to construct a sequence of positive interpolatory quadrature rules by consecutively adding a single node to the rule. There are various possibilities to alleviate this, e.g.\ by allowing negative weights, relaxing the strict requirement that all nodes of the quadrature rule have to be preserved, or by adding multiple nodes instead of one. In this article, the second and third options are further considered. For this purpose, a quadrature rule is presented based on the \emph{replacement} of nodes. The rule has positive weights and is interpolatory, but is strictly speaking not fully nested. The details are considered in Section~\ref{subsubsec:partial}. The addition of multiple nodes is further discussed in Section~\ref{sec:extensionsn}.

\subsubsection{Patterson extension}
\label{subsubsec:patterson1}
Remarkably, both the addition and replacement of a node can yield a Patterson extension of a quadrature rule. In both cases, the focus is on the nodes that yield a zero weight upon addition to the quadrature rule.

In Section~\ref{subsubsec:addition1} it was noticed that any weight from a quadrature rule can be made equal to zero by exploiting the relation between $\vareps_{N+1}$ and $x_{N+1}$. In Example~\ref{ex:addition2} the quadrature rule $X_N = \{-1, -1/6, 1\}$ was considered, where the nodes $-5/3$, $0$, and $7/9$ are such that upon adding one of these to the rule, a rule of only three nodes with non-zero weights of degree three is obtained. Notice that these nodes are Patterson extensions of quadrature rules (as discussed in Section~\ref{subsec:accuracy}), as they can be interpreted as adding one node ($M=1$) to a quadrature rule of two nodes ($N=1$), obtaining a rule of degree three ($N+2M=3$). This also holds in general: for given $k$, adding one node $x^{[k]}_{N+1}$ from \eqref{eq:intervalbdd1} (so $M=1$) to the interpolatory quadrature rule $X_N \setminus \{ x_k \}$ (with degree $N-1$) yields a quadrature rule with $N+1$ nodes and degree $N+1$ (which equals $(N-1)+2M$).

In Section~\ref{subsubsec:replacement1} the notation $x_{(k,l)}$ was introduced to denote nodes that, upon adding them to the rule, yield a (possibly negative) interpolatory quadrature rule with $w^{(N+1)}_k = w^{(N+1)}_l = 0$. These nodes also form a Patterson extension. To see this, notice that the replacement is adding a single node to the quadrature rule $X_{N-2} = X_N \setminus \{ x_k, x_l \}$. The Patterson extension of a single node of this quadrature rule is a quadrature rule consisting of $N$ nodes of degree $(N-2)+2M = N$ (adding one node means $M=1$). By construction, this rule has the nodes $X_{N-2} \cup \{ x_{(k,l)} \}$.

\begin{example}
	Reconsider for example the quadrature rule with the nodes $X_N = \{-1, -1/6, 1\}$ and $\rho \equiv 1/2$. Then it is straightforward to determine using \eqref{eq:xkl} that $x_{(0,1)} = -1/3$, $x_{(0,2)} = 2$, and $x_{(1,2)} = 1/3$. Hence these are three Patterson extensions of the quadrature rule nodes $\{ 1 \}$, $\{ -1/6 \}$, and $\{ -1 \}$. Indeed, the quadrature rules with the nodes $\{ -1/3, 1 \}$, $\{ -1/6, 2 \}$, or $\{ -1, 1/3 \}$ have degree equal to 2.
\end{example}

Notice that $x_{(k,l)}$ is not a Patterson extension of the quadrature rule that has been used to determine it, i.e.\ $X_N$, $W_N$ in \eqref{eq:xkl}. However, its definition allows for a straightforward way to determine this extension. First, add (randomly) two nodes to the quadrature rule $X_N$, $W_N$, obtaining a possibly negative interpolatory quadrature rule $X_{N+2}$, $W_{N+2}$. Then the node $x_{(N+1,N+2)}$ is the Patterson extension of the quadrature rule with nodes $X_N$, because upon adding this node to $X_{N+2}$, the weights of the randomly added nodes become zero. As the Patterson extension is unique, this construction is well-defined. Naturally, this is not the preferred approach to construct a Patterson extension, but it embeds such extensions into the framework discussed here.

The Patterson extension is also obtained as a special case if multiple nodes are added to a quadrature rule. This will be discussed in Section~\ref{subsec:constructionn}.

\subsubsection{Partially nested, positive, and interpolatory quadrature rule}
\label{subsubsec:partial}
\begin{figure}
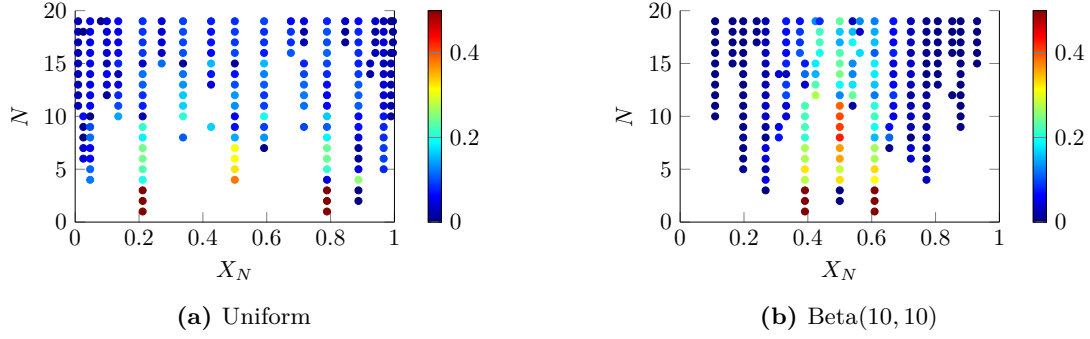

	\begin{minipage}{.5\textwidth}
		\centering
		\includepgf{.8\textwidth}{.6\textwidth}{naive-uniform.tikz}
		\subcaption{Uniform}
		\label{subfig:naiveuniform}
	\end{minipage}%
	\begin{minipage}{.5\textwidth}
		\centering
		\includepgf{.8\textwidth}{.6\textwidth}{naive-beta1010.tikz}
		\subcaption{$\operatorname{Beta}(10, 10)$}
		\label{subfig:naivebeta}
	\end{minipage}
	\caption{Partially nested, positive, and interpolatory quadrature rules constructed using sequences of Gaussian quadrature rules. The $N$-th quadrature rule is constructed by iteratively replacing all nodes of a Gaussian quadrature rule of $N$ nodes by the plotted quadrature rule of $N-1$ nodes. The procedure is initiated using the Gaussian quadrature rule consisting of two nodes. The colors indicate the weights of the nodes.}
	\label{fig:naive}
\end{figure}
The addition and replacement of a single node are straightforward procedures described as the solutions of linear inequalities. However, there does not always exist a single node that can be added such that all weights remain positive. In this section, this is alleviated by relaxing the requirement that $X_N \subset X_{N+M}$.

To this end, let $X_N$ and $\hat{X}_{N+1}$ be the nodes of two positive interpolatory quadrature rules, possibly with $X_N \not\subset \hat{X}_{N+1}$. The nodes $\hat{X}_{N+1}$ can for example form a Gaussian quadrature rule. The idea is to iteratively replace nodes in $\hat{X}_{N+1}$ with nodes from $X_N$, i.e.\ removing $x_k \in \hat{X}_{N+1}$ and adding $x_k \in X_N$. Ideally, all nodes $x_k \in X_N$ can be added to $x_k \in \hat{X}_{N+1}$, which would yield a rule that reuses all nodes in $X_N$.

In other words, if $X_N = \{ x_0, \dots, x_N \}$ and $\hat{X}_{N+1} = \{ \hat{x}_0, \dots, \hat{x}_{N+1} \}$, for each node $x_k \in X_N \setminus \hat{X}_{N+1}$ the set $\Omega_j$ is identified (see Section~\ref{subsubsec:replacement1}) such that $(\hat{X}_{N+1} \cup \{ x_k \}) \setminus \{ \hat{x}_j \}$ is the nodal set of a positive and interpolatory quadrature rule. If there is an $x_k$ such that $\hat{x}_j \not\in X_N$, we set $\hat{X}_{N+1} \gets (\hat{X}_{N+1} \cup \{ x_k \}) \setminus \{ \hat{x}_j \}$ and keep repeating this procedure until no such $x_k$ exists anymore. If there are multiple $x_k$ that could possibly be used to trigger a replacement in $\hat{X}_{N+1}$, the smallest one is selected in the example presented in this article.

The nodes from $X_N$ that cannot be added to $\hat{X}_{N+1}$ are reconsidered in consecutive iterations and added again if possible. It is difficult to theoretically quantify the number of nodes from $X_N$ that can be ``added'' this way to $\hat{X}_{N+1}$, though it is straightforward to see that there exists at least a single $x_k \in X_N$ that can be reused.

To demonstrate this procedure numerically, let $X_1$ and $W_1$ form a Gaussian quadrature rule of two nodes. If the uniform distribution is considered, evaluating all quadrature rules up to $N = 19$ requires in total $22$ unique evaluations of $u$, which is two more than optimally possible considering the limitations of the framework as discussed in this work. The obtained sequence is depicted in Figure~\ref{subfig:naiveuniform} (the two additional evaluations of $u$ can be found at $N = 15$ and $N = 18$). This result seems to be somewhat independent from the distribution, since applying the same approach to construct a sequence of quadrature rules with respect to a $\operatorname{Beta}(10, 10)$ distribution requires in total $23$ function evaluations, which is three more than optimally possible (the obtained rules are depicted in Figure~\ref{subfig:naivebeta}).

The main advantage of this approach compared to the previously discussed Patterson extension is that it always has positive weights. Moreover the expressions to compute the nodes contained in the quadrature rule are straightforward. However, the approach has the same disadvantage as the removal of nodes (see Section~\ref{subsec:removal}), since it requires a sequence of existing quadrature rules.

\section{Addition of multiple nodes}
\label{sec:extensionsn}
In the previous section a counterexample of a positive interpolatory quadrature rule is discussed that can not be extended by adding a single node. In this section we will therefore study the addition of multiple nodes to a quadrature rule. The problem setting is that of Section~\ref{subsec:probsetting}: given a positive interpolatory quadrature rule $X_N$, $W_N$, determine $M$ as small as possible and nodes $X_{N+M}$ with $X_N \subset X_{N+M}$ such that $X_{N+M}$ forms the nodes of a positive interpolatory quadrature rule.

The first step is to extend the derivation of Section~\ref{subsec:positivity1} for the addition of multiple nodes. The derivation is again based on Cramer's rule. With the theory that is derived in the upcoming Section~\ref{subsec:positivityn} it is not obvious how nodes can be added to the quadrature rule, but it provides geometrical insight in the location of such nodes with respect to the existing nodes. Again we can derive some non-trivial adjustments one can apply to a quadrature rule. These are discussed in Section~\ref{subsec:adjustmentsn}. Similar to the case of a single node, there is a tight relation with the Patterson extension. In this case, the Patterson extension for general $M$ is recovered. This is discussed in Section~\ref{subsec:constructionn}, including some examples of nested quadrature rules obtained with the theory derived in this section.

\subsection{Positive weight criterion}
\label{subsec:positivityn}
The idea is similar to the derivation of the addition of single node. Let $X_N$ be the initial nodal set and let $M$ be given. The goal is to determine $X_{N+M}$ with $X_N \subset X_{N+M}$ such that it forms the nodal set of a positive interpolatory quadrature rule.

Let $w^{(N)}_k$ for $k = 0, \dots, N$ be the weights of $W_N$ and likewise let $w^{(N+M)}_k$ be the (unknown) weights of $W_{N+M}$. Then there exists a vector $\mathbf{c} = \trans{(c_0, \dots, c_N, c_{N+1}, \dots, c_{N+M})}$ such that $w_k^{(N+M)} = w_k^{(N)} + c_k$. The goal is to construct $\mathbf{c}$ such that the obtained rule is interpolatory and positive.

With a similar reasoning as before it is straightforward to observe that the following should hold for such a vector to ensure that the obtained quadrature rule is interpolatory:
\begin{equation}
	\sum_{k=0}^{N+M} x_k^j c^{}_k = 0, \text{ for $j = 0, \dots, N$},
\end{equation}
and
\begin{equation}
	\sum_{k=0}^{N+M} x_k^j c^{}_k = \varepsilon_j, \text{ for $j = N+1, \dots, N+M$},
\end{equation}
where $\varepsilon_j$ is as previously introduced, i.e.\ $\varepsilon_j \coloneqq \mu_j - \sum_{k=0}^N x_k^j w^{(N)}_k$. This can be written in the form of a linear system as follows:
\begin{equation}
	\begin{pmatrix}
		x_0^0 & \cdots & x_N^0 & x_{N+1}^0 & \cdots & x_{N+M}^0 \\
		\vdots & \ddots & \vdots & \vdots & \ddots & \vdots \\
		x_0^N & \cdots & x_N^N & x_{N+1}^N & \cdots & x_{N+M}^N \\
		x_0^{N+1} & \cdots & x_N^{N+1} & x_{N+1}^{N+1} & \cdots & x_{N+M}^{N+1} \\
		\vdots & \ddots & \vdots & \vdots & \ddots & \vdots \\
		x_0^{N+M} & \cdots & x_N^{N+M} & x_{N+1}^{N+M} & \cdots & x_{N+M}^{N+M}
	\end{pmatrix}
	\begin{pmatrix}
		c_0 \\
		\vdots \\
		c_N \\
		c_{N+1} \\
		\vdots \\
		c_{N+M}
	\end{pmatrix}
	=
	\begin{pmatrix}
		0 \\
		\vdots \\
		0 \\
		\varepsilon_{N+1} \\
		\vdots \\
		\varepsilon_{N+M}
	\end{pmatrix}.
\end{equation}

Applying Cramer's rule to this system requires more bookkeeping, as the right hand side contains multiple non-zero entries. Let $\varepsb = \trans{(0, \dots, 0, \varepsilon_{N+1}, \dots, \varepsilon_{N+M})}$, then Cramer's rule prescribes
\begin{equation}
	c_k = \frac{\det V_k(X_{N+M})}{\det V(X_{N+M})},
\end{equation}
where $V_k(X_{N+M})$ is equal to $V(X_{N+M})$ with the $k$-th column (indexed from 0) replaced by $\varepsb$. The numerator can be further expanded as follows:
\begin{equation}
	\det V_k(X_{N+M}) = \sum_{j={N+1}}^{N+M} (-1)^{(j+1) + (k+1)} \varepsilon_j \det V_{(j,k)}(X_{N+M}) = \sum_{j={N+1}}^{N+M} (-1)^{j+k} \varepsilon_j \det V_{(j,k)}(X_{N+M}),
\end{equation}
where $V_{(j,k)}(X_{N+M})$ is the $(j, k)$-minor of $V(X_{N+M})$ (i.e.\ the matrix without its $j$-th row and $k$-th column, where both indices start at 0). Hence for $c_k$ the following expression is obtained:
\begin{align}
	c_k &= \sum_{j=N+1}^{N+M} (-1)^{j+k} \varepsilon_j \frac{\det V_{(j,k)}(X_{N+M})}{\det V(X_{N+M})} \\
	&= \left.\sum_{j=N+1}^{N+M} (-1)^{N+M-j} \varepsilon_j \frac{\det V_{(j,k)}(X_{N+M})}{\det V_{(N+M,k)}(X_{N+M})} \middle/ \left( \prod_{\substack{j = 0 \\ j \neq k}}^{N+M} (x_k - x_j) \right)\right..
	\label{eq:detratio}
\end{align}
The same derivation is commonly used to derive the determinant of a Vandermonde matrix~\cite{Gautschi1962,Macon1958}, and it is well-known that the ratio of determinants obtained in this expression is an elementary symmetric polynomial. The $k$-th elementary symmetric polynomial is generally defined as the sum of all monomial permutations of length $k$, that is as follows:
\begin{equation}
	e_k(x_0, \dots, x_N) = \quad \sum_{\mathclap{0 \leq i_1 < \cdots < i_k \leq N}} \quad x_{i_1} \cdots x_{i_k}.
\end{equation}

The elementary symmetric polynomials are only defined for $k \leq N+1$ and by convention $e_0 \equiv 1$. Concluding, the following expression is obtained for $c_k$:
\begin{equation}
	c_k = \left( \sum_{j = N+1}^{N+M} (-1)^{N+M-j} \varepsilon_j e_{N+M-j}(X_{N+M} \setminus \{ x_k \}) \middle) \middle/ \middle( \prod_{\substack{j = 0 \\ j \neq k}}^{N+M} (x_k - x_j) \right), \text{ for $k = 0, \dots, N+M$}.
\end{equation}
Here, $e_k$ is the $k$-th elementary symmetric polynomial as defined above. With a little abuse of notation, we used:
\begin{align}
	e_{N+M-j}(X_{N+M} \setminus \{ x_k \}) \coloneqq{}& e_{N+M-j}(x_0, \dots, x_{k-1}, 0, x_{k+1}, \dots, x_{N+M}) \\
	={}& e_{N+M-j}(x_0, \dots, x_{k-1}, x_{k+1}, \dots, x_{N+M}).
\end{align}
We are now in a position to formulate a theorem in similar form as Theorem~\ref{thm:uni1node}, but then for multiple nodes. The proof is omitted, since it is equivalent to that of Theorem~\ref{thm:uni1node}, but then with the equalities derived in this section.

\begin{theorem}
	\label{thm:uninnode}
	Let $X_N$, $W_N$ form an interpolatory quadrature rule. Then $X_{N+M} = X_N \cup \{ x_{N+1}, \dots, x_{N+M} \}$ forms the nodal set of a positive interpolatory quadrature rule if and only if
	\begin{equation}
		\label{eq:uninnode}
		- \left( \sum_{j = N+1}^{N+M} (-1)^{N+M-j} \varepsilon_j e_{N+M-j}(X_{N+M} \setminus \{ x_k \}) \middle) \middle/ \middle( \prod_{\substack{j = 0 \\ j \neq k}}^{N+M} (x_k - x_j) \right) \leq w^{(N)}_k, \text{ for $k = 0, \dots, N+M$}.
	\end{equation}
\end{theorem}

For $M=1$, we have that the summation only incorporates $j = N+1$, hence $e_{N+M-j}(X_{N+M} \setminus \{ x_k \}) = e_0(X_{N+M} \setminus \{ x_k \}) = 1$, recovering Theorem~\ref{thm:uni1node}. So Theorem~\ref{thm:uninnode} is indeed a strict generalization of Theorem~\ref{thm:uni1node}.

\subsection{Quadrature rule adjustments}
\label{subsec:adjustmentsn}
Theorem~\ref{thm:uninnode} presents a necessary and sufficient condition for a quadrature rule extended with $M$ nodes to have positive weights. Contrary to the addition of a single node, it cannot be used directly to determine possible nodes that can be added to the quadrature rule. This can be seen by rewriting it in a similar form as \eqref{eq:posineq1}, i.e.\ for $k = 0, \dots, N+M$:
\begin{equation}
\label{eq:posineqn1}
\begin{aligned}
	w^{(N)}_k \prod_{\substack{j = 0 \\ j \neq k}}^{N+M} (x_k - x_j) &\geq -\sum_{j = N+1}^{N+M} (-1)^{N+M-j} \varepsilon_j e_{N+M-j}(X_{N+M} \setminus \{ x_k \}) &\text{if } \prod_{\substack{j = 0 \\ j \neq k}}^{N+M} (x_k - x_j) \geq 0, \\
	w^{(N)}_k \prod_{\substack{j = 0 \\ j \neq k}}^{N+M} (x_k - x_j) &\leq -\sum_{j = N+1}^{N+M} (-1)^{N+M-j} \varepsilon_j e_{N+M-j}(X_{N+M} \setminus \{ x_k \}) &\text{if } \prod_{\substack{j = 0 \\ j \neq k}}^{N+M} (x_k - x_j) \leq 0.
\end{aligned}
\end{equation}
Notice that, if $x_{N+1}, \dots, x_{N+M}$ are unknowns, an $M$-variate system of $N+M+1$ polynomial inequalities is obtained. In general these systems are very difficult to solve, so we do not directly pursue a solution of the system above. Nonetheless, the system still provides a geometrical interpretation about where solutions reside, similar to the case of single node addition (though less intuitive). This is discussed in Section~\ref{subsubsec:geometryn}. Based on these geometrical insights, procedures to replace nodes and to add nodes, which extend those explained previously, can be derived.  These procedures are discussed in Section~\ref{subsubsec:replacementn} and~\ref{subsubsec:additionn} respectively.

\subsubsection{Geometry of nodal addition}
\label{subsubsec:geometryn}
The type of the inequalities \eqref{eq:posineqn1} (i.e.\ ``greater than'' versus ``less than'') does not change between two nodes and if this type is fixed, the system consists of polynomial inequalities. Hence the region where $M$ nodes can be added is described by a continuous boundary, bounded by the polynomial inequalities of \eqref{eq:posineqn1}, consisting of lines, surfaces, or ``hypersurfaces'' through the nodes.

If one of the right hand sides of \eqref{eq:posineqn1} changes sign, there is an addition of $M$ nodes such that the inequality forms an equality for a specific $k$. In such cases, there is an addition such that one of the nodes obtains a weight equal to zero. This is equivalent to the case discussed in Section~\ref{subsubsec:replacement1}, where a single node is added in order to set the weights of another node equal to zero.

It is difficult to visualize the addition of $M$ nodes in a similar way as we visualized the addition of one node, as there are $M$ nodes $x_{N+1}, \dots, x_{N+M}$ and $M$ quadrature rule errors $\vareps_{N+1}, \dots, \vareps_{N+M}$. Plotting the errors with respect to the nodes (as in Figure~\ref{fig:additionreplacement}) is therefore not viable, as this is a plot from $\mathbb{R}^M$ to $\mathbb{R}^M$.

On the other hand, if the distribution $\rho(x)$ is fixed beforehand, the values of $\vareps_{N+1}, \dots, \vareps_{N+M}$ are known and contour plots of the regions encompassing all $M$ nodes that can be added can be made (provided that $M$ is small enough).

\begin{example}
	Let $\rho \equiv 1/2$ with $\Omega = [-1, 1]$ and reconsider the quadrature rule from Example~\ref{ex:addition}. In Figure~\ref{fig:addition-2d} lines are depicted where the inequalities from \eqref{eq:posineqn1} are equalities. The shaded area depicts regions where all inequalities are valid, i.e.\ any coordinate $(x_{N+1}, x_{N+2})$ in the shaded region can be added to the respective quadrature rule in order to obtain a positive interpolatory rule. The figure is obviously symmetric around $x_{N+1} = x_{N+2}$, as the order of addition (i.e.\ first adding $x_{N+1}$ and then $x_{N+2}$ or vice versa) yields equivalent quadrature rules. Selecting a coordinate $(x_{N+1}, x_{N+2})$ on one of the boundaries results into one weight equal to zero. Adding the coordinates on the corners, depicted by the open circles (i.e.\ ``the boundary of the boundary''), results into two weights equal to zero.

	The dashed lines indicate where the inequalities \eqref{eq:posineqn1} with $k = N+1$ and $k = N+2$ change sign. If this happens, one of the new nodes $x_{N+1}$ or $x_{N+2}$ has weight equal to zero. This line forms everywhere a boundary of the shaded area: the node with weight equal to zero can be replaced by any other node, while still resulting into an interpolatory quadrature rule with positive weights. This situation is equivalent to adding a single node $x_{N+1}$ to the quadrature rule, but gaining two degrees, as discussed in Section~\ref{subsubsec:addition1}.
\end{example}

\begin{figure}
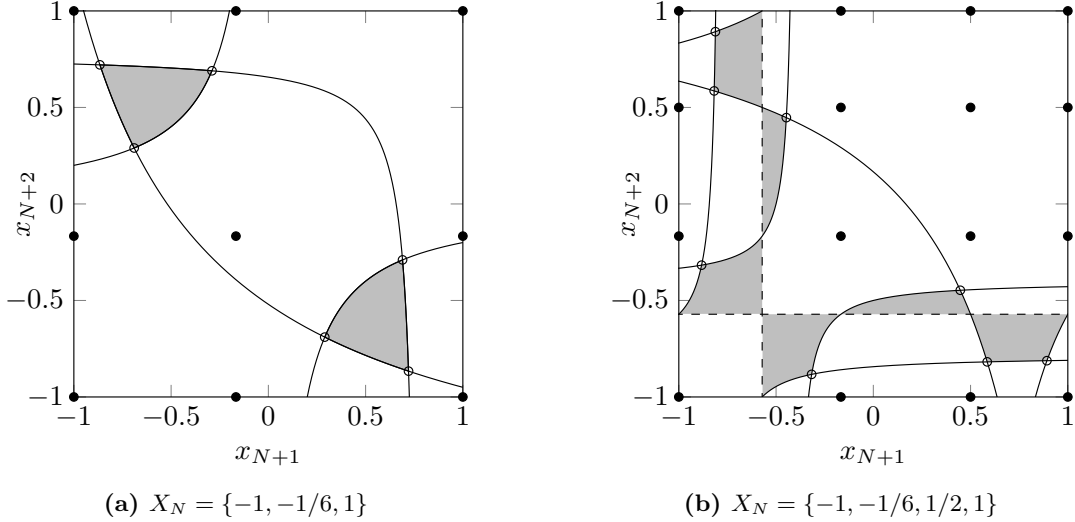

	\centering
	\begin{minipage}{.5\textwidth}
		\centering
		\includepgf{.8\textwidth}{.8\textwidth}{addition-2d-1.tikz}
		\subcaption{$X_N = \{-1, -1/6, 1\}$}
		\label{subfig:addition-2d-1}
	\end{minipage}%
	\begin{minipage}{.5\textwidth}
		\centering
		\includepgf{.8\textwidth}{.8\textwidth}{addition-2d-2.tikz}
		\subcaption{$X_N = \{-1, -1/6, 1/2, 1\}$}
		\label{subfig:addition-2d-2}
	\end{minipage}
	\caption{Two examples of addition of two nodes to a quadrature rule. In both cases, $\rho \equiv 1/2$. Choosing the two nodes in a shaded area yields positive weights. Choosing the two nodes on the open circles yields two weights equal to zero and positive weights. Dashed lines correspond to a zero weight for $x_{N+1}$ or $x_{N+2}$, i.e.\ adding $x_{N+1}$ exactly at the dashed line yields a quadrature rule of $N+2$ nodes with degree $N+2$, making the addition of $x_{N+2}$ trivial.}
	\label{fig:addition-2d}
\end{figure}

The addition and replacement of multiple nodes follow readily from this example. Notice that if any coordinate $(x_{N+1}, \dots, x_{N+M})$ is known, the replacement for $M=1$ can be used to reach any other coordinate $(x_{N+1}, \dots, x_{N+M})$ in the same region (shaded in Figure~\ref{fig:addition-2d}). Hence if all corners of those regions are determined (depicted as open circles in Figure~\ref{fig:addition-2d}), the full region can be explored straightforwardly using Algorithm~\ref{alg:replacement2}. As these corner cases form a replacement of nodes, we start by discussing replacement of $M$ nodes. Moreover, it will be shown that these corners are a Patterson extension. Based on the algorithm to determine all these corners, addition of $M$ nodes follows straightforwardly.

\subsubsection{Replacement of multiple nodes}
\label{subsubsec:replacementn}
Let $X_N$, $W_N$ be an interpolatory quadrature rule and let indices $k_1, \dots, k_M$ be given such that $0 \leq k_i \leq N$ and $k_i \neq k_j$ for $i \neq j$. In this section the goal is to determine the interpolatory quadrature rule $X_{N+M}$, $W_{N+M}$ such that $w^{(N+M)}_{k_i} = 0$ for all $k_i$. Notice that this is equivalent to replacing the nodes $x_{k_1}, \dots, x_{k_M}$ in the quadrature rule $X_N$ by the nodes $x_{N+1}, \dots, x_{N+M}$. The nodes with this property are the intersections of the polynomials of \eqref{eq:posineqn1} and they are depicted as open circles in Figure~\ref{fig:addition-2d}. Moreover, they describe the boundary of the set of nodes that can be added to the quadrature rule.

The desired nodes $x_{N+1}, \dots, x_{N+M}$ can be determined by calculating the Patterson extension of the interpolatory quadrature rule with the nodes $X_N \setminus \{x_{k_1}, \dots, x_{k_M}\}$, for which efficient techniques exist~\cite{Patterson1989,Laurie1997,Laurie1996}. Such techniques require that $M$ must be known a priori and they do not provide a simple geometrical interpretation. Therefore we proceed by embedding the Patterson extension in the framework discussed here. This yields an alternative, new algorithm to determine these nodes, which is mainly of theoretical and geometrical interest, since it requires the computation of large numbers of roots of polynomials.

We start by solving a slightly easier problem. Assume $\vareps_{N+1} = \cdots = \vareps_{N+M-1} = 0$ and $\vareps_{N+M} \neq 0$. Notice that, if $\vareps_{N+M}$ is neglected, any addition of $M-1$ nodes yields a valid quadrature rule (as these nodes have zero weight). Geometrically, a fully shaded figure (if drawn as Figure~\ref{fig:addition-2d}) is obtained. This can be exploited to determine the desired nodes, as only the value of $\vareps_{N+M}$ imposes a condition on the nodes $x_{N+1}, \dots, x_{N+M}$.

The nodes that yield $w^{(N+M)}_{k_1} = \dots = w^{(N+M)}_{k_M} = 0$ can be found by applying Theorem~\ref{thm:uninnode} with $c_{k_i} = -w^{(N)}_{k_i}$ for all $i$ or by consecutively applying Theorem~\ref{thm:uni1node}. In both cases, the following is obtained:
\begin{equation}
	\label{eq:nullifyM}
	\vareps_{N+M} = -w^{(N)}_{k_i} \left(\prod_{\substack{j = 0 \\ j \neq k_i}}^N (x_{k_i} - x_j)\middle) \middle(\prod_{j={N+1}}^{N+M} (x_{k_i} - x_j)\right), \text{ for $i = 1, \dots, M$}.
\end{equation}
In principle this system of polynomial equalities is difficult to solve, but it has a certain structure that can be exploited. To see this, let $\hat\ell_M(x)$ be the nodal polynomial of the nodes $x_{N+1}, \dots, x_{N+M}$:
\begin{equation}
	\hat\ell_M(x) = \prod_{j={N+1}}^{N+M} (x - x_j),
\end{equation}
which translates the system above to
\begin{equation}
	\label{eq:systemell}
	\vareps_{N+M} = -w^{(N)}_{k_i} \left(\prod_{\substack{j = 0 \\ j \neq k_i}}^N (x_{k_i} - x_j)\right) \hat\ell_M(x_{k_i}), \text{ for $i = 1, \dots, M$}.
\end{equation}
If the nodal polynomial $\hat\ell_M$ is known, its roots equal $x_{N+1}, \dots, x_{N+M}$. The nodal polynomial has degree $M$ and it is known that its leading order coefficient equals 1. Therefore it is useful to introduce the polynomial $q_M(x) \coloneqq \hat\ell_M(x) - x^M$, which has degree $M-1$. Then \eqref{eq:systemell} can be rewritten as follows:
\begin{equation}
	\label{eq:defell}
	q_M(x_{k_i}) = \hat\ell_M(x_{k_i}) - x_{k_i}^M = -\left. \vareps_{N+M} \middle/ \middle(w^{(N)}_{k_i} \prod_{\substack{j = 0 \\ j \neq k_i}}^N (x_{k_i} - x_j)\right) - x_{k_i}^M, \text{ for $i = 1, \dots, M$}.
\end{equation}
These are $M$ values of a polynomial of degree $M-1$, which is a well-known interpolation problem and can be solved with various well-known methods (such as barycentric interpolation~\cite{Berrut2004}). If $q_M$ is determined, the roots of the polynomial $\hat\ell_M(x) = q_M(x) + x^M$ are the nodes $x_{N+1}, \dots, x_{N+M}$. By construction these nodes are such that $w^{(N+M)}_{k_i} = 0$ for $i = 1, \dots, M$.

Even though assuming $\vareps_{N+1} = \cdots = \vareps_{N+M-1} = 0$ is not realistic in practical cases, this procedure can readily be extended to the general case. For this we reuse the replacement step from Section~\ref{subsubsec:replacement1}. If $\vareps_{N+1} \neq 0$, then a single node is added to the quadrature rule such that $w^{(N+1)}_{k_1} = 0$. This is equivalent to applying Algorithm~\ref{alg:replacement1} with $x_{N+1} = x_{(k,l)}$, as discussed in Section~\ref{subsubsec:replacement1}. Then the obtained quadrature rule $X_{N+1} \setminus \{ x_{k_1} \}$ has $\vareps_{N+1} = 0$. By applying the procedure discussed above to these $N+1$ nodes, the nodes $x_{N+2}$ and $x_{N+3}$ can be determined such that $w^{(N+2)}_{k_2} = 0$ and $w^{(N+2)}_{N+1} = 0$, i.e.\ we enforce that the weight of $x_{k_2}$ is zero and the weight of the previously added node becomes zero. The obtained rule has $N+3$ nodes, where two nodes have weight equal to zero. This is again a replacement, but here \emph{two} nodes get weight equal to zero, which is a generalization of the replacement discussed in Section~\ref{subsubsec:replacement1}. Those nodes are removed to reobtain a quadrature rule of $N+1$ nodes and this process is repeated iteratively until $X_{N+M}$ is obtained. The obtained rule can be interpreted as a replacement of $M$ nodes, and yields the open circles from Figure~\ref{fig:addition-2d}. It is an iterative description: a replacement of $M$ nodes is determined using a replacement of $M-1$ nodes. Geometrically, we iterate over the dimension of the figure and iteratively determine a set of nodes that can be used as a replacement.

The obtained nodes form by definition a Patterson extension of the nodal set $X_N \setminus \{ x_{k_1}, \dots, x_{k_M} \}$, since it holds that $(X_N \setminus \{ x_{k_1}, \dots, x_{k_M} \}) \cup \{ x_{N+1}, \dots, x_{N+M} \}$ has degree $N+M$. The existence of such a Patterson extension is directly coupled to the existence of $M$ nodes that can possibly be added to $X_N$ in the hope of obtaining an interpolatory quadrature rule with positive weights: if $M$ nodes can be added to the quadrature rule, the Patterson extension has positive weights, since it forms the boundary of the set that describes all additions. Moreover, if all Patterson extensions of all sets $X_N \setminus \{ x_{k_1}, \dots, x_{k_M} \}$ for any sequences $(k_1, \dots, k_M)$ have negative weights or are not real-valued, no addition of $M$ nodes exists.

Hence we have proved the following lemma.

\begin{lemma}
	\label{lmm:patterson}
	Let $X_N$, $W_N$ form a positive interpolatory quadrature rule, let $\rho$ (or a sequence of moments) be the density function, and let $M$ be given. Then the following statements are equivalent:
	\begin{enumerate}
		\item There exists a Patterson extension of $M$ nodes of the quadrature rule $X_N$, $W_M$ with solely non-negative weights;
		\item There exist $M$ nodes $x_{N+1}, \dots, x_{N+M}$ such that $X_N \cup \{ x_{N+1}, \dots, x_{N+M} \}$ forms the nodal set of a positive interpolatory quadrature rule.
	\end{enumerate}
\end{lemma}

As stated before, any algorithm that computes Patterson extensions can be used to verify whether $M$ nodes exist that can be added to the rule. If a Patterson extension with non-negative weights is found, say $x_{N+1}, \dots, x_{N+M}$, Algorithm~\ref{alg:replacement2} can be used to explore all possible additions to the quadrature rule.

The algorithm based on the geometrical interpretation used in this article is outlined in Algorithm~\ref{alg:mtuple}. By iterating over all possible sorted sequences $(k_1, \dots, k_M)$, this procedure can be used straightforwardly to verify whether there exist $M$ nodes that can be added to a given quadrature rule (though this is a costly procedure). 

\begin{algorithm}[t]
\caption{Determining $X_{N+M}$ with zero weights}
\label{alg:mtuple}
\begin{algorithmic}[1]
\Require Interpolatory quadrature rule $X_N$, $W_N$, indices $k_1, \dots, k_M$.
\Ensure Interpolatory quadrature rule $X_{N+M}$, $W_{N+M}$ such that $w^{(N+M)}_{k_i} = 0$ for all $i$.

~

\State $m \gets 1$
\For{$k = k_1, \dots, k_M$}
	\State Determine $\hat\ell_m$ such that $\hat\ell_m(x) = x^m + q_m(x)$ (see \eqref{eq:defell})\label{eqref} and
	\[
		\vareps_{N+m} = -w_l \, \hat\ell_m(x) \prod_{\substack{j = 0 \\ j \neq l}}^{N+m} (x_l - x_j) \text{ for both $l = k$ and $l = N+1, \dots, N+m-1$}
	\]
	\State Let $r_1, \dots, r_m$ be the roots of $\hat\ell_m$, i.e.\ $\hat\ell_m(r_k) = 0$
	\State $X_{N+m} \gets X_N \cup \{ r_1, \dots, r_m \}$ and determine $W_{N+m}$
	\State $m \gets m+1$
\EndFor
\State \textbf{Return} $X_{N+M}$, $W_{N+M}$
\end{algorithmic}
\end{algorithm}

There are two special cases that are (for sake of simplicity) not incorporated in Algorithm~\ref{alg:mtuple}. Firstly, if $w^{(N+m)}_k = 0$ at the start of an iteration, the polynomial $\hat\ell_M(x)$ is not well-defined. This can be incorporated by selecting any non-zero $w^{(N+m)}_{k_i}$ at the start of the iteration. If no such $w^{(N+m)}_{k_i}$ exists, then all these weights are zero, which is the primary goal of the algorithm. Secondly, if $r_k \in X_N$ or $\vareps_{N+m} = 0$, a quadrature rule is obtained that has higher degree than its number of nodes. This can be incorporated by combining all double nodes in $X_N$ and likewise adding the respective weights and by skipping any iteration that has $\vareps_{N+m} = 0$.

\subsubsection{Addition of multiple nodes}
\label{subsubsec:additionn}
By combining the quadrature rule replacement of Section~\ref{subsubsec:replacement1} (for $M=1$) and the replacement of the previous section (for $M > 1$), we obtained a naive algorithm to firstly determine $M$ as small as possible such that there exists a positive interpolatory quadrature rule $X_{N+M}$ (Algorithm~\ref{alg:mtuple}) and secondly to explore \emph{all} such $M$ nodes (Algorithm~\ref{alg:replacement2}, yielding the shaded areas of Figure~\ref{fig:addition-2d}).

Determining the number of nodes $M$ that can be added to an interpolatory quadrature rule can straightforwardly be done by solving \eqref{eq:nullifyM} for each sequence of $k_1, \dots, k_M$ with $k_1 < \cdots < k_M$. This gives all locations where $M$ nodes have zero weight. If at any of these locations all nodes have non-negative weight, then $M$ nodes can be added to the rule. Otherwise, $M$ is increased and the process is repeated.

Often the value of $M$ is unknown a priori. Besides determining the $M$ nodes that can be added, the goal is also to determine $M$ as small as possible (this is also how we formulated the problem originally in Section~\ref{subsec:probsetting}). Algorithm~\ref{alg:mtuple} can be used to determine $M$, as results from previous iterations can be reused. To see this, suppose a quadrature rule is given and by applying Algorithm~\ref{alg:mtuple} it is known that no addition of at most $M-1$ nodes exist. Then during these calculations, all sequences of nodes have been determined that make $M-1$ weights zero. By initializing Algorithm~\ref{alg:mtuple} with these sequences, only the last iteration of the loop is necessary, which significantly reduces the computational expense.

It is required to repeatedly determine large numbers of polynomial roots in this algorithm. This is nearly impossible to do symbolically, except for some special cases (e.g.\ $M \leq 3$ or symmetric quadrature rules). Moreover determining the roots numerically can result in quick aggregation of numerical errors. We use variable precision arithmetic, i.e.\ determine the roots with a large number of significant digits.

For large $N$ this is a costly algorithm, as the number of sorted sequences of length $M$ equals
\begin{equation}
	\#(k_1, \dots, k_M) = \binom{N+1+M}{M},
\end{equation}
which grows fast for large $N$. Therefore using this algorithm to compute all removals is slower than using existing techniques to compute the Patterson extension, albeit that it is able to reuse all additions of $M-1$ nodes to compute all additions of $M$ nodes.

If all sets of $M$ nodes have been determined that can be added to the quadrature rule, the techniques from Section~\ref{subsubsec:replacement1} can be used to fully explore all nodes that can be added to the rule. This requires solving linear equalities, which can be done fast and accurately.

The possibility of adding $M$ nodes to the quadrature rule does not guarantee the possibility of adding $M+1$ nodes to the quadrature rule.

\begin{example}
	We revisit the quadrature rule example from Example~\ref{ex:addition}, i.e.
	\begin{equation}
		X_N = \left\{ -1, -\frac{1}{6}, 1 \right\}, W_N = \left\{ \frac{1}{10}, \frac{24}{35}, \frac{3}{14} \right\}.
	\end{equation}
	In Figure~\ref{subfig:addition-2} regions are depicted where a single node can be added (similar to Figure~\ref{subfig:addition}) and regions where, upon adding a node from that region, another node can be added (this is the projection of Figure~\ref{subfig:addition-2d-1}). The addition of the rightmost node with the latter property is depicted in Figure~\ref{subfig:addition-2-after}, demonstrating that there is a single node that can be added and that this is indeed a limiting case.
\end{example}

Notice that the intervals where a single node and where two nodes can be added are independent from each other. There exist pairs of nodes $x_{N+1}, x_{N+2}$ firstly such that both $W_{N+1}$ and $W_{N+2}$ are all positive (in the right interval surrounded by squares), secondly such that $W_{N+1}$ is positive, but $W_{N+2}$ is not (the right interval surrounded by circles, outside the interval surrounded by squares), thirdly such that $W_{N+1}$ is not positive, but $W_{N+2}$ is (the left interval surrounded by squares), and finally such that both $W_{N+1}$ and $W_{N+2}$ are always negative (outside all intervals).

\begin{figure}
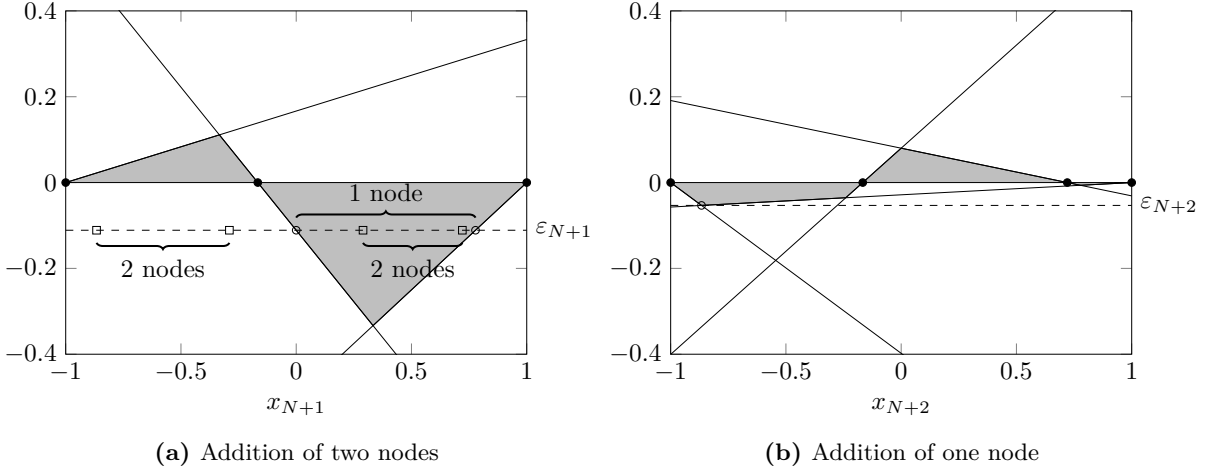

	\begin{minipage}{.5\textwidth}
		\centering
		\includepgf{\textwidth}{.8\textwidth}{addition-2.tikz}
		\subcaption{Addition of two nodes}
		\label{subfig:addition-2}
	\end{minipage}%
	\begin{minipage}{.5\textwidth}
		\centering
		\includepgf{\textwidth}{.8\textwidth}{addition-2-after.tikz}
		\subcaption{Addition of one node}
		\label{subfig:addition-2-after}
	\end{minipage}
	\caption{The addition of 2 nodes to the interpolatory quadrature rule with the nodes $X_N = \{-1, -1/6, 1\}$. \emph{Left:} intervals depicting which node to select if the goal is to add one or two nodes to the quadrature rule. Hence selecting any node between the two squares and adding it yields a quadrature rule to which again a node can be added. The interval of adding a single node is the same as depicted in Figure~\ref{subfig:addition}. \emph{Right:} The quadrature rule obtained by adding the rightmost highlighted node of the left figure (i.e.\ ``the rightmost square''). Hence there is only a single node that can be added to the rule.}
	\label{fig:addition-2}
\end{figure}

\subsection{Constructing quadrature rules}
\label{subsec:constructionn}
Similar to the case of addition of a single node, the Patterson extension is obtained for specific choices of nodes that are added to the rule. In fact, the nodes determined with Algorithm~\ref{alg:mtuple} are a Patterson extension of a quadrature rule with a smaller number of nodes. As the Gaussian quadrature rule is a special case of the Patterson extension, this rule also follows from the framework discussed in this article. This is discussed in more detail in Section~\ref{subsubsec:pattersonn}.

By repeatedly applying Algorithm~\ref{alg:mtuple}, a sequence of nested quadrature rules can be determined. These rules and their properties are considered in Section~\ref{subsubsec:nested}.

\subsubsection{Patterson extension}
\label{subsubsec:pattersonn}
The boundary of the set that describes all possible additions is spanned by the Patterson extension (the open circles in Figure~\ref{subfig:addition} and Figure~\ref{subfig:addition-2d-1}). These nodes have the property that, upon adding them to the quadrature rule, a rule of degree $N+M$ is obtained with $M$ weights equal to zero. This is equivalent to the Patterson extension of the quadrature rule \emph{without} those $M$ nodes with zero weight. For $M=1$, this was demonstrated in Section~\ref{subsubsec:patterson1}.

For general $M$, the Patterson extension can be deduced mathematically as follows. Let $X_N$, $W_N$ be a quadrature rule and, as before, let $x_{N+1}, \dots, x_{N+M}$ be such that the following nodes form a quadrature rule of degree $N+M$:
\begin{equation}
	\label{eq:nodes1}
	(X_N \cup \{ x_{N+1}, \dots, x_{N+M} \}) \setminus \{ x_{k_1}, \dots, x_{k_M} \}.
\end{equation}
Furthermore, let $X_{N-M}$ be the nodes of an interpolatory quadrature rule of degree $N-M$ be as follows:
\begin{equation}
	\label{eq:nodes2}
	X_{N-M} = X_N \setminus \{ x_{k_1}, \dots, x_{k_M} \}.
\end{equation}
Upon adding $\{ x_{N+1}, \dots, x_{N+M} \}$ to $X_{N-M}$, the nodes from \eqref{eq:nodes1} are obtained, that have degree $N+M$. Hence $M$ nodes are added to an interpolatory rule of degree $N-M$ and the obtained degree is $N+M$, which is by definition a Patterson extension. Notice that the obtained quadrature rule is interpolatory, but not necessarily positive.

The Gaussian quadrature rule can be deduced as special case from Algorithm~\ref{alg:mtuple}. To see this, suppose $M = N+1$, which is the number of nodes of the rule under consideration. In that case, there is only a single sequence of $k_1, \dots, k_M$, defined as follows up to a permutation:
\begin{equation}
	k_j = j-1 \text{ for $j = 1, \dots, N+1$}.
\end{equation}
By applying Algorithm~\ref{alg:mtuple}, the nodes from \eqref{eq:nodes1} are obtained with $M=N+1$, which are:
\begin{equation}
	(X_N \cup \{ x_{N+1}, \dots, x_{2N+1} \}) \setminus \{ x_0, \dots, x_N \} = \{ x_{N+1}, \dots, x_{2N+1} \}.
\end{equation}
Hence the $N+1$ nodes $x_{N+1}, \dots, x_{2N+1}$ form a quadrature rule of degree $2N+1$, which is by definition the Gaussian quadrature rule. In other words, when adding a Gaussian quadrature rule to an existing quadrature rule and setting all existing weights to zero a valid addition is obtained.

\begin{example}
	To demonstrate where Patterson extensions occur in our work, reconsider the interpolatory quadrature rule with the nodes $X_N = \{ -1, -1/6, 1 \}$. In Section~\ref{subsubsec:patterson1} three different Patterson extensions related to this quadrature rule were discussed: $\{ -1/3, 1 \}$, $\{-1/6, 2\}$, or $\{ -1, 1/3 \}$. All these rules are Patterson extensions (of smaller quadrature rules) with $M=1$. To obtain a Patterson extension with $M=2$ and subsequently a Gaussian quadrature rule, consider Algorithm~\ref{alg:mtuple} using $\{k_1, k_2, k_3\} = \{0, 1, 2\}$. The algorithm proceeds as follows:
	\begin{enumerate}
		%

		\item In the first iteration, it follows that $\hat\ell_1(x) = x + 5/3$ and therefore the following quadrature rule is obtained:
		\begin{equation}
			X_{N+1} = \left\{ -1, -\frac{1}{6}, 1, -\frac{5}{3} \right\}, W_{N+1} = \left\{ 0, \frac{16}{21}, \frac{11}{56}, \frac{1}{24} \right\}. 
		\end{equation}
		Notice that the node $x_{N+1} = -5/3$ was obtained in Section~\ref{subsubsec:addition1}, where we discussed that after adding this node one obtains $w_0^{(3)} = 0$.

		\item In the second iteration, it follows that $\hat\ell_2(x) = x^2 + 2/5 x - 1/5$. Here, the Patterson extension with $M=2$ of the quadrature rule with ``nodes'' $\{ 1 \}$ is obtained. Hence the following rule is obtained (notice that the node $-5/3$ is removed):
		\begin{align}
			X_{N+2} &= \left\{ -1, -\frac{1}{6}, 1, \frac{1}{5} \left( -1 - \sqrt{6} \right), \frac{1}{5} \left( -1 + \sqrt{6} \right) \right\}, \\
			W_{N+2} &= \left\{ 0, 0, \frac{1}{9}, \frac{1}{36} \left( 16 + \sqrt{6} \right), \frac{1}{36} \left( 16 - \sqrt{6} \right) \right\}.
		\end{align}

		\item In the third iteration, it follows that $\hat\ell_3(x) = x^3 - 3/5 x$, whose roots are the Gaussian quadrature rule or, equivalently, the Patterson extension with $M=3$ of the empty quadrature rule:
		\begin{equation}
			X_{N+3} = \left\{ -1, -\frac{1}{6}, 1, -\frac{1}{5} \sqrt{15}, 0, \frac{1}{5} \sqrt{15} \right\}, W_{N+3} = \left\{ 0, 0, 0, \frac{5}{18}, \frac{4}{9}, \frac{5}{18} \right\}.
		\end{equation}
	\end{enumerate}
	In this specific example it is possible to determine all nodes symbolically, but for larger values of $M$ this is generally not possible.
\end{example}

Considering the nodes in a different order results into different intermediate Patterson extensions, but obviously the Gaussian quadrature rule is the rule that is finally obtained. These steps also demonstrate the possibility to store intermediate results: only the nodes of step 2 are necessary to deduce the nodes of step 3.

Specialized algorithms exist for specific distributions and specific values of $N$ and $M$ to construct Gaussian, Gauss--Kronrod, and Gauss--Patterson quadrature rules~\cite{Golub1969,Laurie1997}, but it remains a challenging topic to determine the Patterson extension for general non-Gaussian quadrature rules. The algorithm presented in this article is not an alternative for these existing algorithms, but embeds the Patterson extension in the discussed framework and can be used to determine \emph{all} $M$ nodes that can be added to a quadrature rule. If an efficient procedure to determine large numbers of Patterson extensions is available, it can be readily used to determine whether an extension for a specific $M$ exists. By consecutively replacing the new nodes (see Section~\ref{subsubsec:replacement1}) all $M$ nodes that can be added can be found.

\subsubsection{Nested, positive, and interpolatory quadrature rule}
\label{subsubsec:nested}
\begin{figure}
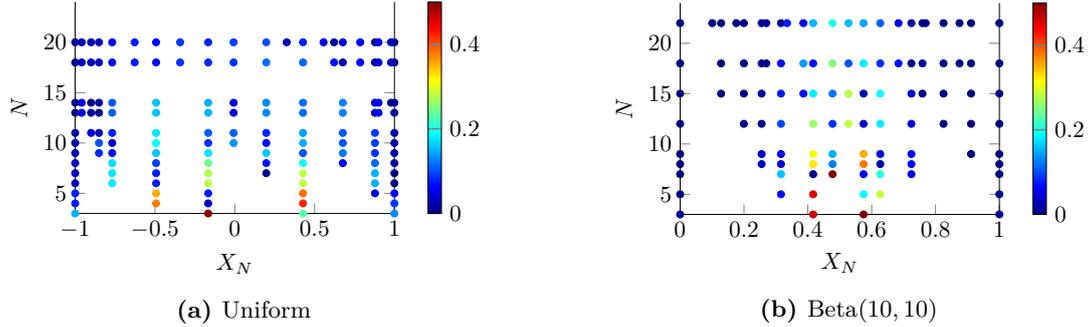

	\begin{minipage}{.5\textwidth}
		\centering
		\includepgf{.8\textwidth}{.6\textwidth}{full-uniform.tikz}
		\subcaption{Uniform}
		\label{subfig:fulluniform}
	\end{minipage}%
	\begin{minipage}{.5\textwidth}
		\centering
		\includepgf{.8\textwidth}{.6\textwidth}{full-beta1010.tikz}
		\subcaption{$\operatorname{Beta}(10, 10)$}
	\end{minipage}
	\caption{Nested, positive, and interpolatory quadrature rules, initialized with $X_N = \{-1, -1/6, 1\}$ (left) or $X_N = \{0, 5/12, 1\}$ (right). Given the $N$-th quadrature rule, the next rule is obtained by firstly computing the minimal number of nodes that can be added and by secondly randomly adding such a minimal number of nodes to the rule. The colors indicate the weights of the nodes.}
	\label{fig:full}
\end{figure}

Algorithm~\ref{alg:mtuple} provides a straightforward procedure to determine the minimal value of $M$ and the positive interpolatory quadrature rule nodes $X_{N+M}$ such that $X_N \subset X_{N+M}$. The replacement procedure for $M=1$ of Section~\ref{subsubsec:replacement1} can be used to determine all possible nodes, given $M$. This is the original goal of the article as outlined in Section~\ref{subsec:probsetting} and examples of such quadrature rules are depicted in Figure~\ref{fig:full}. Here, each quadrature rule is iteratively extended with a minimal number of nodes, and the nodes that are added are selected randomly from the set containing all $M$ nodes that can be added. There are two main differences with the quadrature rules obtained in Section~\ref{subsubsec:partial}, where an existing rule was used as basis for a larger quadrature rule: the rules obtained in this section are fully nested, but do add more than one node between two consecutive rules.

Both figures demonstrate that $M$ varies significantly and does not increase monotonically. This is in line with the conclusions drawn in the Section~\ref{subsubsec:additionn}, as shown in Figure~\ref{fig:addition-2}. Moreover for almost all $N$, the value of $M$ is significantly larger in case the Beta distribution is considered, which is related to the ``bad'' initial set of nodes for this distribution. A different initialization would lead to different values of $M$.

\section{Numerical integration with positive quadrature rules}
\label{sec:numerics}
This article is concerned with the construction of quadrature rules with \emph{positive} weights and two new quadrature rules have been introduced: one based on the consecutive replacement of single nodes (possibly resulting in a sequence of rules that is not nested) and one by randomly adding nodes ensuring positive weights. We briefly assess the numerical performance of these quadrature rules by means of the Genz test functions (see Table~\ref{tbl:genz1d}). The Genz test functions~\cite{Genz1984} are functions defined on $\Omega = [0, 1]$ constructed specifically to test integration routines. Each function has a specific family attribute that is considered to be challenging for integration routines, that can be enlarged by a shape parameter $a$ and translated by a translation parameter $b$. We restrict ourselves to the uniform distribution, as in this case the exact value of the integral of the Genz functions is known analytically.

We consider the performance of the following four quadrature rules:
\begin{enumerate}
	\item A quadrature rule that is determined by consecutively adding and replacing nodes originating from a Gaussian quadrature rule (see Figure~\ref{subfig:naiveuniform}). This rule was discussed in Section~\ref{subsec:construction1} and is a partially nested, positive, and interpolatory quadrature rule. The rule is initialized with the quadrature rule nodes $X_N = \{0, 5/12, 1\}$ (i.e.\ the nodes from the example as discussed before, translated to $[0, 1]$).

	\item A quadrature rule that is determined by consecutively randomly adding $M$ nodes to the rule such that the obtained rule is positive. Here $M$ is minimal, i.e.\ the smallest number of nodes is added for each $N$ (see Figure~\ref{subfig:fulluniform}). This rule was discussed in Section~\ref{subsec:constructionn} and is a nested, positive, and interpolatory quadrature rule. The rule is initialized in the same way as the quadrature rule of the previous point, i.e.\ using $X_N = \{0, 5/12, 1\}$.

	\item The Clenshaw--Curtis quadrature rule~\cite{Clenshaw1960}, where the nodes $X_N$ are defined explicitly by \eqref{eq:clenshaw}. It is well known that these nodes have positive weights if the distribution under consideration is uniform, which is the case. This positive and interpolatory quadrature rule is nested for specific levels, i.e.\ $X_{N_L} \subset X_{N_{L+1}}$ with $N_L = 2^L$ ($l = 1, 2, \dots$).

	\item The Gaussian quadrature rule~\cite{Golub1969}, where the nodes and weights are defined as the quadrature rule with $N+1$ nodes of degree $2N+1$. This quadrature rule is not nested, so refining the quadrature rule results in a significant number of new function evaluations.
\end{enumerate}

\begin{table}
	\caption{The test functions from Genz~\cite{Genz1984}, which depend on the shape and translation parameters $a$ and $b$.}
	\centering
	\begin{tabular}{l l}
		\textbf{Integrand Family} & \textbf{Attribute} \\
		\hline
		\hline
		$u_1(x) = \cos\left(2\pi b + a x\right)$ & Oscillatory \\
		$u_2(x) = \left(a^{-2} + (x - b)^2\right)^{-1}$ & Product Peak \\
		$u_3(x) = \left(1 + a x\right)^{-2}$ & Corner Peak \\
		$u_4(x) = \exp\left(-a^2 (x - b)^2 \right)$ & Gaussian \\
		$u_5(x) = \exp\left(-a |x - b|\right)$ & $C^0$ function \\
		$u_6(x) = \begin{cases}
			0 &\text{if $x > b$} \\
			\exp\left(a x\right) &\text{otherwise}
		\end{cases}$ & Discontinuous
	\end{tabular}
	\label{tbl:genz1d}
\end{table}

The error measure $e_N$ is the absolute integration error, i.e.
\begin{equation}
	e_N(u) = | \mathcal{I} u - \mathcal{A}_N u |,
\end{equation}
where $u = u_g$ with $g = 1, \dots, 6$, i.e.\ $u$ is one of the Genz test functions. To obtain meaningful results we select the parameters $a$ and $b$ randomly in the unit interval and repeat the experiment 100 times. This also affects the reduced quadrature rule: each experiment selects the node that is removed randomly and therefore 100 different sequences of nested quadrature rules are obtained. The errors reported here are averaged over the 100 experiments and are therefore denoted by $\overline{e}_N$.

It is instructive to compare the error with the upper bound that follows from the Lebesgue inequality \eqref{eq:lebesgue}:
\begin{equation}
	\label{eq:lebesgue2}
	e_N(u) \leq 2 ~ \inf_{\mathclap{\varphi \in \mathbb{P}(N)}} ~ \| u - \varphi \|_\infty,
\end{equation}
where we use that $\mu_0 = 1$ in our test cases. This error is determined using the algorithm of Remez~\cite[Chapter~3]{Watson1980}, with the implementation from \texttt{chebfun}~\cite{Driscoll2014}. Convergence results for the uniform distribution $\rho \equiv 1$ in $\Omega = [0, 1]$ are gathered in Figure~\ref{fig:genz-uniform}.

\begin{figure}
	\centering
	\resizebox{\textwidth}{!}{\includegraphics{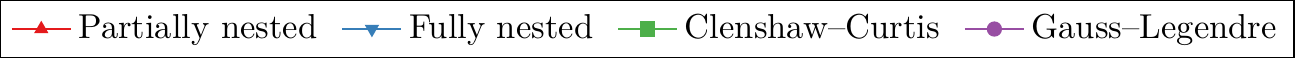}}

	\bigskip

	\begin{minipage}[t]{.5\textwidth}
		\centering
		\includepgf{.9\textwidth}{.8\textwidth}{genz-1.tikz}
		\subcaption{$u_1$}
	\end{minipage}%
	\begin{minipage}[t]{.5\textwidth}
		\centering
		\includepgf{.9\textwidth}{.8\textwidth}{genz-2.tikz}
		\subcaption{$u_2$}
	\end{minipage}

	\begin{minipage}[t]{.5\textwidth}
		\centering
		\includepgf{.9\textwidth}{.8\textwidth}{genz-3.tikz}
		\subcaption{$u_3$}
	\end{minipage}%
	\begin{minipage}[t]{.5\textwidth}
		\centering
		\includepgf{.9\textwidth}{.8\textwidth}{genz-4.tikz}
		\subcaption{$u_4$}
	\end{minipage}

	\begin{minipage}[t]{.5\textwidth}
		\centering
		\includepgf{.9\textwidth}{.8\textwidth}{genz-5.tikz}
		\subcaption{$u_5$}
	\end{minipage}%
	\begin{minipage}[t]{.5\textwidth}
		\centering
		\includepgf{.9\textwidth}{.8\textwidth}{genz-6.tikz}
		\subcaption{$u_6$}
	\end{minipage}
	
	\caption{Convergence of the Genz test functions using various quadrature rule techniques. The absolute error of the best approximation polynomial (i.e.\ $\inf_{\varphi \in \Phi_N} \| u - \varphi \|_\infty$) is dashed.}
	\label{fig:genz-uniform}
\end{figure}

Notice that regardless of the function under consideration all quadrature rule errors remain far under the dashed line, that represents the right-hand side of \eqref{eq:lebesgue2}. This shows that the bound from this inequality is far from sharp.

The first four Genz functions can be approximated well using polynomials, as they are analytic and have rapidly converging Chebyshev coefficients. The best approximation converges exponentially in these cases, which is also the case for the four quadrature rules under consideration. The quadrature rules determined using the framework of this article perform slightly worse than the Clenshaw--Curtis and the Gaussian quadrature rule. This is related to the fact that these rules exploit the structure of the underlying distribution to a large extent (e.g.\ symmetry and higher-order moments), whereas the rules in this work only optimize for the positivity of the weights. The Gaussian quadrature rule converges with the highest rate, which is related to its high polynomial degree (a rule of $N+1$ nodes has degree $2N+1$). However, the Gaussian rule is not nested, so to refine the estimate of the integral for increasing number of nodes the number of function evaluations increases significantly. If a computationally expensive function is considered, using a nested quadrature rule with fine granularity (such as the proposed rules) significantly reduces the cost of refining the quadrature rule estimate.

The fifth Genz test function is not differentiable and can therefore not be approximated well using a polynomial. This can be observed from the best approximation polynomial, that converges with order 1 (so we would expect that $\overline{e}_N \sim 1/N$). In this case the difference between the Gaussian rule and the other rules is significantly smaller, demonstrating that the high polynomial degree of Gaussian rules is less relevant if the integrand is not smooth.

The sixth Genz test function cannot be approximated accurately using a polynomial when considering the $\infty$-norm, as it is discontinuous. Hence the best approximation error remains constant. However, the approximation of the quadrature rules still converges with order $1/2$. In this case, there is a clear difference between the integration error (that is an averaged error) and the best approximation error (that is a uniform error).

\section{Conclusion}
\label{sec:conclusion}
In this article, a novel mathematical framework is presented for the construction of nested, positive, and interpolatory quadrature rules by using a geometrical interpretation. Given an existing quadrature rule, necessary and sufficient conditions have been derived for $M$ new nodes to form an interpolatory quadrature rule with positive weights. The conditions have been formulated as inequalities, which are explicit if $M=1$ and implicit if $M>1$.

The addition of a single node can be treated as a special case, which can be solved analytically. The analytical expression can be used to add nodes to and replace nodes within a quadrature rule. The addition of multiple nodes can be determined numerically and a naive algorithm is presented for this purpose. Based on the quadrature rules obtained by this algorithm, the set that encompasses all additions of $M$ nodes can be explored by iteratively replacing nodes.

The well-known Patterson extension of quadrature rules forms a special case of the framework, as it is obtained by constructing the quadrature rules with $M$ weights equal to zero. As such, our proposed framework and its geometrical interpretation are well embedded in existing theory on the addition of nodes to quadrature rules. The framework provides various possibilities to construct or adapt quadrature rules and two examples have been discussed: one based on consecutively adding and replacing one node and one based on consecutively adding multiple nodes.

Numerical integration using the two quadrature rules introduced in this work shows the key advantages of nested quadrature rules with positive weights: estimates computed using the quadrature rules are stable and nesting allows for computationally cheap refinements of the estimates. Existing quadrature rules, such as the Gaussian and the Clenshaw--Curtis quadrature rule, are not nested with the fine granularity as the rules in this work.

There are various options to further extend the framework set out in this article. The algorithm to determine whether multiple nodes exist that can be added to the quadrature rule depends on determining many polynomial roots and iterates over all possible sequences of nodes that can become zero. For a large number of nodes this is computationally very costly and therefore warrants the need to derive an efficient algorithm to determine these nodes. Moreover the framework set out in this article does not use the relations that exists between consecutive moments of a distribution~\cite{Shohat1943}, which can possibly be used to further extend the framework set out in this text.

\section*{Acknowledgments}
This research is part of the Dutch EUROS program, which is supported by NWO domain Applied and Engineering Sciences and partly funded by the Dutch Ministry of Economic Affairs.

\bibliographystyle{plainnatnourl}
\bibliography{literature.bib}

\end{document}